\documentclass[11pt,reqno]{amsart} 
\usepackage[utf8]{inputenc}
\usepackage{amsfonts,amsmath,amsthm,amssymb,fullpage}
\usepackage[justification=centering]{caption}
\usepackage{color}
\usepackage{enumitem}   
\usepackage{hyperref}
\usepackage{mathtools}
\usepackage{xcolor}
\usepackage{thm-restate}
\usepackage{thmtools}
\usepackage{caption}
\usepackage{subcaption}

\usepackage[backend=biber,sorting=nty,style=numeric,maxbibnames=99]{biblatex}
\makeatletter
\newcommand{\addresseshere}{%
  \enddoc@text\let\enddoc@text\relax
}
\makeatother

\DeclareMathOperator*{\argmax}{arg\,max}

\newtheorem{theorem}{Theorem}[section]
\newtheorem{proposition}[theorem]{Proposition}
\newtheorem{lemma}[theorem]{Lemma}
\newtheorem{corollary}[theorem]{Corollary}
\newtheorem{remark}[theorem]{Remark}
\newtheorem{conjecture}[theorem]{Conjecture}
\theoremstyle{definition}
\newtheorem{example}[theorem]{Example}
\newtheorem{definition}[theorem]{Definition}
\newtheorem{openproblem}{Open Problem}
\newtheorem{claim}{Claim}
\newenvironment{cproof}{\paragraph{\emph{Proof.}}}{\hfill$\square$}
\usepackage{etoolbox}
\AtEndEnvironment{proof}{\setcounter{claim}{0}}
\newtheoremstyle{notation}
{\topsep}
{\topsep}
{}
{}
{\bfseries}
{.}
{.5em} 
{} 
\theoremstyle{notation}

\newcommand*{\y}{\mathbf{y}}
\newcommand*{\x}{\mathbf{x}}
\newcommand*{\z}{\mathbf{z}}

\newcommand*{\PS}{\mathsf{PS}}
\newcommand*{\PF}{\mathsf{PF}}
\newcommand*{\PFND}{\mathsf{PF}^{\uparrow}}

\newcommand*{\PSINV}{\mathsf{PS}^{\mathrm{inv}}}

\newcommand*{\N}{\mathbb{N}}
\newcommand*{\PA}{\mathsf{PA}}
\newcommand*{\PAINV}{\mathsf{PA}^{\mathrm{inv}}}
\newcommand*{\PAINVND}{\mathsf{PA}^{\mathrm{inv},\uparrow}}

\usepackage[colorinlistoftodos]{todonotes}

\title{
Permutation invariant parking assortments
}

\author[Chen]{Douglas M. Chen}
\address[D.~M.~Chen]{Department of Mathematics, Johns Hopkins University, Baltimore, MD 21218}
\email{\textcolor{blue}{\href{mailto:dchen101@jhu.edu}{dchen101@jhu.edu}}}

\author[Harris]{Pamela E. Harris}
\address[P.~E. Harris]{Department of Mathematical Sciences, University of Wisconsin-Milwaukee, Milwaukee, WI 53211}
\email{\textcolor{blue}{\href{mailto:peharris@uwm.edu}{peharris@uwm.edu}}}
\thanks{P.~E.~Harris was supported through a Karen Uhlenbeck EDGE Fellowship.}

\author[Mart\'{i}nez Mori]{J. Carlos Mart\'{i}nez Mori}
\address[J.~C.~Mart\'{i}nez Mori]{Schmidt Science Fellows}
\email{\textcolor{blue}{\href{mailto:jmartinezmori@schmidtsciencefellows.org }{jmartinezmori@schmidtsciencefellows.org}}}

\author[Pab\'{o}n-Cancel]{Eric J. Pab\'{o}n-Cancel}
\address[E.~J.~Pab\'{o}n-Cancel]{Department of Mathematics, Purdue University, West Lafayette, IN 47907}
\email{\textcolor{blue}{\href{mailto:epabonca@purdue.edu}{epabonca@purdue.edu}}}

\author[Sargent]{Gabriel Sargent}
\address[G. Sargent]{Department of Mathematics, University of Notre Dame, Notre Dame, IN 46556}
\email{\textcolor{blue}{\href{mailto:gsargent@nd.edu}{gsargent@nd.edu}}}


\begin{document}

\maketitle

\begin{abstract}
We introduce \emph{parking assortments}, a generalization of parking functions with cars of assorted lengths.
In this setting, there are $n\in\mathbb{N}$ cars of lengths $\mathbf{y}=(y_1,y_2,\ldots,y_n)\in\mathbb{N}^n$ entering a one-way street with $m=\sum_{i=1}^ny_i$ parking spots.
The cars have parking preferences $\mathbf{x}=(x_1,x_2,\ldots,x_n)\in[m]^n$, where $[m]:=\{1,2,\ldots,m\}$, and enter the street in order.
Each car $i \in [n]$, with length $y_i$ and preference $x_i$, follows a natural extension of the classical parking rule: it begins looking for parking at its preferred spot $x_i$ and parks in the first $y_i$ contiguously available spots thereafter, if there are any.
If all cars are able to park under the preference list $\mathbf{x}$, we say $\mathbf{x}$ is a parking assortment for $\mathbf{y}$. 
Parking assortments also generalize \emph{parking sequences}, introduced by Ehrenborg and Happ, since each car seeks for the first contiguously available spots it fits in past its preference.
Given a parking assortment $\mathbf{x}$ for $\mathbf{y}$, we say it is \emph{permutation invariant} if all rearrangements of $\mathbf{x}$ are also parking assortments for $\mathbf{y}$.
While all parking functions are permutation invariant, this is not the case for parking assortments in general, motivating the need for characterization of this property.
Although obtaining a full characterization for arbitrary $n\in\mathbb{N}$ and $\mathbf{y}\in\mathbb{N}^n$ remains elusive, we do so for $n=2,3$.
Given the technicality of these results, we introduce the notion of \emph{minimally invariant} car lengths, for which the only invariant parking assortment is the all-ones preference list.
We provide a concise, oracle-based characterization of minimally invariant car lengths for any $n\in\mathbb{N}$. 
Our results around minimally invariant car lengths also hold for parking sequences.
\end{abstract}

\section{Introduction}
\label{section: introduction}

Parking functions were introduced by Konheim and Weiss in their study of hashing functions~\cite{konheim1966occupancy}.
We describe parking functions as follows.
Consider a one-way street with $n \in \N:=\{1,2,3,\ldots\}$ parking spots.
There are $n$ cars waiting to enter the street sequentially, each of which has a preferred spot.
When a car enters the street, it attempts to park in its preference.
If its preferred spot is occupied, the car continues driving down the street until it finds an unoccupied spot in which to park (if there is one).
If in this process a car does not find such a spot, we say parking fails.
If the preference list allows all cars to park, we say it is a \emph{parking function of length $n$}.
Let $\PF_n$ denote the set of parking functions of length $n$. It is well-known that \begin{align}
|\PF_n|=(n+1)^{n-1},\label{eq:enum pfs}
\end{align}
for a proof we refer the reader to \cite[Lemma~1]{konheim1966occupancy}.
Parking functions have been subject to much recent interest, from their variants, specializations, and generalizations~\cite{carlson2021parking, harris2023outcome}, to polyhedral combinatorics~\cite{benedetti2019combinatorial,amanbayeva2022convex}, and to their connections with other research areas such as sorting algorithms~\cite{harris2023lucky}, the game of Brussels Sprouts~\cite{ji2021brussels}, and the combinatorics of partially ordered sets~\cite{elder2023boolean}, to list a few.

Ehrenborg and Happ introduce a generalization of parking functions, known as \emph{parking sequences}, in which cars have assorted (integer) lengths~\cite{ehrenborg2016parking}. 
Let $\y = (y_{1}, y_{2}, \ldots, y_{n}) \in \N^n$ be a list of car lengths, where $y_i$ denotes the length of car $i \in [n] := \{1, 2, \ldots, n\}$. 
Consider a one-way street with $m = \sum_{i=1}^n y_i$ parking spots.
We encode the cars' preferences via the preference list $\x = (x_1, x_2, \ldots, x_n) \in [m]^n$.
In this scenario, car $i$ fails to park if, upon its arrival, the first available spot $j \geq x_i$ is such that either $j+y_i-1>m$ or at least one of the subsequent spots $j+1,\ldots,j+y_i-1$ is already occupied (this is referred to as a ``collision'' in the literature).
If all cars are able to park under the preference list $\x$, then $\x$ is said to be a \emph{parking sequence} (of length $n$) for $\y$.
For a fixed $\y$, let $\PS_n(\y)$ denote the set of parking sequences\footnote{
Although the length $n$ of the elements in $\PS_n(\y)$ is implied by the number of entries in $\y$, we include the subscript $n$ for clarity, as it keeps track of the number of cars.
} for $\y$.

For arbitrary $\y\in\N^n$, in \cite[Theorem~3]{ehrenborg2016parking}, Ehrenborg and Happ show that
\begin{align}
\label{eq: count}
    |\PS_n(\y)| = (y_1 + n) \cdot (y_1 + y_2 + n - 1) \cdots (y_1 + y_2 + \cdots + y_{n - 1} + 2).
\end{align}
Note that the set of parking functions of length $n$ is the same as $\PS_n(\y)$ in the special case where $\y = (1^n) := (1, 1, \ldots, 1) \in \N^n$. 
Thus, in the case where $\y=(1^n)$, Equation~\eqref{eq: count} reduces 
to Equation~\eqref{eq:enum pfs}.

In this work, we introduce a new set called \emph{parking assortments}. 
As in parking sequences, there are 
$n\in\mathbb{N}$ cars of assorted lengths $\mathbf{y}=(y_1,y_2,\ldots,y_n)\in\mathbb{N}^n$ entering a one-way street containing $m=\sum_{i=1}^ny_i$ parking spots.
The cars have parking preferences $\mathbf{x}=(x_1,x_2,\ldots,x_n)\in[m]^n$ and enter the street in order.
Each car $i \in [n]$, with length $y_i$ and preference $x_i$, follows a natural extension of the classical parking rule: it begins looking for parking at its preferred spot $x_i$ and parks in the first $y_i$ contiguously available spots thereafter (if there are any).
If all cars are able to park under the preference list $\mathbf{x}$, we say that $\mathbf{x}$ is a \emph{parking assortment} for $\mathbf{y}$. 
Note that parking assortments generalize parking functions, which correspond to the case in which $\mathbf{y}$ is a list of all ones. 
Note that parking assortments also generalize parking sequences. 
Namely, in the parking sequence setting a car seeks to park only once by attempting to park once it locates the first available spot at or after its preference, and ``gives up'' if it does not fit. 
However, in the parking assortment case, a car continues to seek parking after its preference, seeking forward  (until the end of the street) for enough available spaces in which to fit.
Thus, given $\mathbf{y}$, the set of parking assortments contains the set of 
parking sequences. 
For a fixed $\y$, let $\PA_n(\y)$ denote its set of parking assortments.

We remark that if $n=1$ or $n=2$, then $\PS_n(\y)=\PA_n(\y)$ for any $\y \in \N^n$, however this fails to hold for $n\geq 3$. 
To illustrate, consider  Figure~\ref{figure: pa vs ps} in which  $\y=(1,3,1)$ and $\x=(2,1,1)$. 
In this case, $\x$ is a parking assortment for $\y$, as the second car can proceed past spot $1$ and fill the last three spots. 
However, $\x$ is not a parking sequence, as the second car will attempt to park in spots $1$, $2$, and $3$, resulting in a ``collision'' with the first car, which has already occupied spot $2$. 
\begin{figure}[ht]
    \centering
    \begin{subfigure}[h]{0.475\linewidth}
    \centering
    \includegraphics[width=\linewidth]{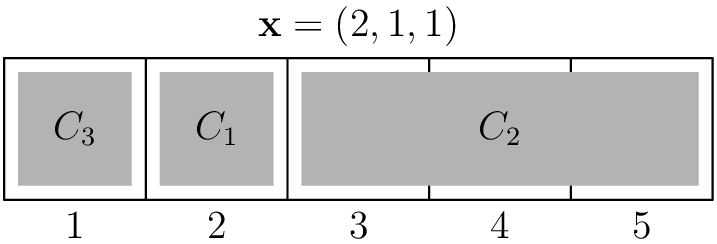}
    \caption{Under assortment parking rules}
    \end{subfigure}
    \hfill 
    \begin{subfigure}[h]{0.475\linewidth}
    \centering
    \includegraphics[width=\linewidth]{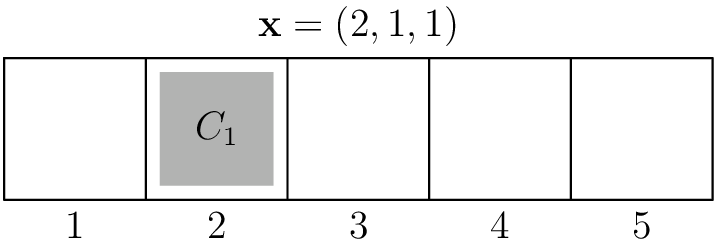}
    \caption{Under sequence parking rules}
    \end{subfigure}    
    \caption{
    If $\y = (1,3,1)$, then $\x=(2,1,1)\in\PA_3(\y)$, but $\x\notin \PS_3(\y)$.}
    \label{figure: pa vs ps}
\end{figure}

Many of the techniques used to study parking functions and their generalizations rely on the fact that any rearrangement of a parking function is itself a parking function~\cite[pp.~836]{yan2015parking}.
However, it is known that this is not the case for parking sequences, and hence also not the case for parking assortments. 
To illustrate this, we present the example in Figure~\ref{figure: example}.
Given car lengths $\y = (1, 2, 2)$, the preference list $\x = (1, 2, 1)$ is a parking sequence and hence a parking assortment, whereas its rearrangement $\x' = (2, 1, 1)$ is neither.
\begin{figure}[ht]
    \centering
    \begin{subfigure}[h]{0.475\linewidth}
    \includegraphics[width=\linewidth]{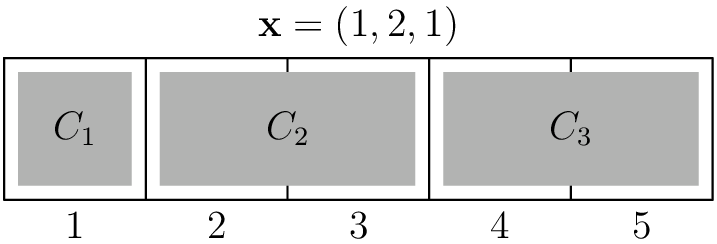}
    \caption{All cars can park}
    \end{subfigure}
    \hfill
    \begin{subfigure}[h]{0.475\linewidth}
    \includegraphics[width=\linewidth]{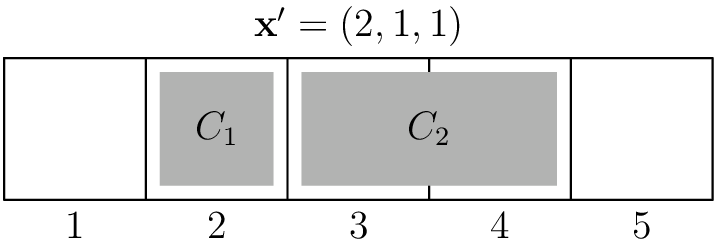}
    \caption{Car 3 fails to park}
    \end{subfigure}
    \caption{
    If $\y = (1, 2, 2)$, then $\x = (1, 2, 1)\in\PA_3(\y)$ and $\x' = (2, 1, 1)\notin \PA_3(\y)$.}
    \label{figure: example}
\end{figure}

This motivates the following questions:
\begin{enumerate}[label=\roman*)]
\item When is a rearrangement of a parking sequence itself a parking sequence?
\item When is a rearrangement of a parking assortment itself a parking assortment?
\end{enumerate}
For certain families of car lengths $\y\in\N^n$, Adeniran and Yan study rearrangements of parking sequences, their characterizations, and enumerations \cite{adeniran2021increasing}.
In this work, we initiate the analogous study of rearrangements of parking assortments. 
Formally, given $\mathbf{y} \in \N^n$ and $\x\in\PA_n(\y)$, we say $\x$ is a \emph{permutation invariant parking assortment for $\y$} whenever all of the rearrangements\footnote{In referencing the entries of a list, we use the words ``permutation'' and ``rearrangement'' synonymously.} of $\x$ are themselves in $\PA_n(\y)$.
We generally omit the word ``permutation'' and say \emph{invariant parking assortments} for short.
For a fixed $\y \in \N^n$, let $\PAINV_n(\y)$ denote the set of invariant parking assortments for $\y$.

We now summarize our contributions.
In Section \ref{sec:properties},  we present some closure properties of parking assortments. In Section \ref{section: constant car lengths}, we consider parking assortments for car lengths of the form $\y=(c^n):=(c,c,\ldots,c)\in\mathbb{N}^n$. Our results include:
\begin{enumerate}
\item A noninductive proof characterizing the elements of $\PAINV_n((c^n))$ (Theorem~\ref{theorem: constant}), which is analogous to ~\cite[Theorem 3.4]{adeniran2021increasing}, characterizing the elements of $\PSINV_n((c^n))$.
\item Establish that the number of nondecreasing parking assortments is given by $C_n = \frac{1}{n+1}\binom{2n}{n}$, the $n$th Catalan number\footnote{OEIS \textcolor{blue}{\href{http://oeis.org/A000108}{A000108}}.} (Corollary~\ref{corollary: catalan count}).
\item Establish that $|\PAINV_n((c^n))|=(n+1)^{n-1}$ (Corollary~\ref{corollary: constant count}). 
\end{enumerate}
In Section~\ref{section: minimally invariant car lengths}, we state and prove our main result (Theorem~\ref{theorem: mi}), which gives necessary and sufficient conditions on the car lengths $\y \in \N^n$ for the set of invariant parking assortments to consist solely of the all-ones preference list (i.e., $\PAINV_n(\y) = \{(1^n)\}$).
We refer to such car lengths $\y$ as \emph{minimally invariant}.
Our characterization is concise in the sense that it is given in the form of $O\left(n \cdot \sum_{i=1}^n y_i\right)$ efficiently-computable checks (i.e., distinct executions of the ``parking experiment'').

Applications of Theorem \ref{theorem: mi} include the following:
\begin{enumerate}
    \item Boolean formula characterizations of minimally invariant car lengths for two and three cars (Corollary~\ref{corollary: mi pair} and Theorem~\ref{theorem: mi triple}, respectively).
    
    \item Establishing the closure property that if $\y \in \N^n$ is minimally invariant, then so is its restriction $\y_{\vert_{i}} := (y_1, y_2, \ldots, y_i)$ for any $i \in [n]$ (Corollary~\ref{corollary: restriction}).
    \item 
    Establishing the property that if $\y \in \N^n$ is strictly increasing, then it is minimally invariant (Theorem \ref{theorem: increasing length}).
\end{enumerate}

In Section~\ref{section: two and three cars} we provide a full characterization of invariant parking assortments with two and three cars, beyond those implied by minimally invariant car lengths  (Theorem~\ref{theorem: invariant pair} and Theorem~\ref{theorem: invariant triple}, respectively).
We conclude in Section~\ref{section: open problems} 
with multiple directions for future study. 

\section{Properties of Parking Assortments}\label{sec:properties}
Recall the following results of Adeniran and Yan on  parking sequences.
\begin{proposition}[[Proposition~2.3 in \cite{adeniran2021increasing}]
\label{proposition: ps nondecreasing}
Let $\y = (y_1, y_2, \ldots, y_n) \in \N^n$. 
Further, let $\x = (x_1, x_2, \ldots, x_n) \in [m]^n$ be nondecreasing.
Then, $\x \in \PS_n(\y)$ if and only if $x_i \leq 1 + \sum_{j=1}^{i-1} y_j$ for all $i \in [n]$.
\end{proposition}

Using the same techniques, we obtain an analogous result for parking assortments.
We utilize this result throughout our work.
\begin{proposition}
\label{proposition: nondecreasing}
Let $\y = (y_1, y_2, \ldots, y_n) \in \N^n$. 
Further, let $\x = (x_1, x_2, \ldots, x_n) \in [m]^n$ be nondecreasing.
Then, $\x \in \PA_n(\y)$ if and only if $x_i \leq 1 + \sum_{j=1}^{i-1} y_j$ for all $i \in [n]$.
\end{proposition}

\begin{proof}
We claim that $\x \in \PA_n(\y)$ implies $x_i \leq 1 + \sum_{j=1}^{i-1} y_j$ for all $i \in [n]$.
To prove this, we assume there exists $i \in [n]$ with $x_i > 1 + \sum_{j=1}^{i-1} y_j$ and show that $\x \notin \PA_n(\y)$.
Pick any such $i$ and note that, since $\x$ is nondecreasing, we have $1 + \sum_{j=1}^{i-1} y_j < x_i \leq x_{i + 1} \leq \cdots \leq x_n$.
Therefore, cars $i, i + 1, \ldots, n$, which collectively require $\sum_{j=i}^n y_j$ spots, can only park in the at most $\sum_{j=1}^{n} y_j - (1 + \sum_{j=1}^{i-1} y_j) = \sum_{j=i}^n y_j - 1$ spots to the inclusive right of spot $x_i$.
This implies that at least one of cars $i, i + 1, \ldots, n$ is unable to park, which is to say $\x \notin \PA_n(\y)$.

Next, we claim that $x_i \leq 1 + \sum_{j=1}^{i-1} y_j$ for all $i \in [n]$ implies $\x \in \PA_n(\y)$.
We prove this by induction on the number of cars that have arrived.
Upon the arrival of car $1$, our assumption that $x_1 \leq 1$ implies car $1$ of length $y_1$ parks in spots $1, 2, \ldots, y_1$.
By way of induction, for $k\in[n-1]$, suppose that cars $1, 2, \ldots, k$ collectively park in spots $1, 2, \ldots, \sum_{i=1}^{k} y_i$.
We show that, upon the arrival of car $k + 1$, it parks in spots $1 + \sum_{i=1}^{k} y_i, 2 + \sum_{i=1}^{k} y_i, \ldots, \sum_{i=1}^{k+1} y_i$.
By the inductive hypothesis, spots $1, 2, \ldots \sum_{i=1}^{k} y_i$ are occupied whereas spots $1 + \sum_{i=1}^{k} y_i, 2 + \sum_{i=1}^{k} y_i, \ldots, \sum_{i=1}^{n} y_i$ are unoccupied.
Then, together with our assumption that $x_{k + 1} \leq 1 + \sum_{i=1}^{k} y_i$, we have that the first $y_{k + 1}$ contiguously unoccupied spots that car $k + 1$ finds upon its arrival are spots $1 + \sum_{i=1}^{k} y_i, 2 + \sum_{i=1}^{k} y_i, \ldots, \sum_{i=1}^{k+1} y_i$, and so those are the spots in which it parks. 
This ultimately shows all cars are able to park, which is to say $\x \in \PA_n(\y)$.
\end{proof}

Next, we show the following technical result showing that invariant parking assortments are closed under the concurrent removal of the last car length from $\y$ and the largest preference from $\x$. 
To state the result, we need the following notation.
For ${\bf{v}} \in \N^n$ and $i \in [n]$, 
let ${\bf{v}}_{\widehat{i}} := (v_1, \ldots, v_{i-1}, \widehat{v}_i, v_{i+1}, \ldots, v_n)$ denote the removal of the $i$th entry of ${\bf{v}}$ and ${\bf{v}}_{\vert_i} := (v_1, v_2, \ldots, v_i)$ denote the restriction of ${\bf{v}}$ to its first $i$ entries.
\begin{lemma}
\label{lemma: removal}
Let $\y \in \N^n$ and $\x = (x_1, x_2, \ldots, x_n) \in \PAINV_n(\y)$.
If $k = \argmax_{i \in [n]} x_i$, then $\x_{\widehat{k}} \in \PAINV_{n-1}(\y_{\vert_{n-1}})$.
\end{lemma}
\begin{proof}
Consider the rearrangement $\x'$ of $\x$ defined as $\x' = (x_1, \ldots, x_{k-1}, x_{k+1}, \ldots, x_{n}, x_k)$, where we move the entry $x_k$ to the last index and shift the entries $x_{k+1},\ldots,x_n$ to the left by one index. 
Note that $\x \in \PAINV_n(\y)$ implies $\x' \in \PAINV_n(\y)$.
In particular, it implies that all cars are able to park under $\x'$.
Given that $k=\argmax_{i \in [n]} x_i$, under $\x'$, car $n$ has the largest preference, so it necessarily parks in the last $y_n$ spots.
Since $\x' \in \PAINV_n(\y)$, membership in $\PAINV_n(\y)$ is invariant under permutation of the first $n-1$ entries of $\x'$.
Therefore, for any rearrangement $\x'' = (x_1'', x_2'', \ldots, x_n'')$ of $\x'$ with $x_n'' = x_k$, all cars are able to park, car $n$ parks in the last $y_n$ spots, and the first $n-1$ cars park in the first $\sum_{i=1}^{n-1} y_i$ spots.
Note that the first $n-1$ entries of any such rearrangement correspond to a rearrangement of $\x_{\widehat{k}}$.
It follows that $\x_{\widehat{k}} \in \PAINV_{n-1}(\y_{\vert_{n-1}})$.
\end{proof}

The following result is utilized in Section \ref{section: two and three cars} where we study invariant parking assortments with two and three cars.
We present it here since it gives a quick test for determining that a preference list is not an invariant parking assortment.
\begin{lemma}
\label{lemma: minentry}
Let $\y=(y_1,y_2,\ldots,y_n) \in \mathbb{N}^n$ and $\x=(x_1,x_2,\ldots,x_n) \in \mathbb{N}^n$. 
Let $k = \min_{i \in [n]} y_i$.
If there exists $i \in [n]$ with $1 < x_i \leq k$, then $\x \notin \PAINV_n(\y)$.
\end{lemma}
\begin{proof}
Consider the rearrangement $\x' =(x_i,x_1,\ldots,x_{i-1},x_{i+1},\ldots,x_n)$ of $\x$ in which we move entry $x_i$ to the first index in $\x'$ and shift entries $x_1,x_2,\ldots,x_{i-1}$ right by one index. 
Then, car $1$ under $\x'$ leaves a gap (of unoccupied spots) of size $x_{i}-1< k$ to the left of spot $x_i$, which no subsequent car can fill.
Therefore, $\x' \notin \PA_n(\y)$, which implies $\x \notin \PAINV_n(\y)$.
\end{proof}

\section{Constant Car Lengths}
\label{section: constant car lengths}
In this section, we consider cars having the same lengths. Our first result gives a characterization for parking assortments analogous to that of parking sequences as established in \cite[Theorem 3.4]{adeniran2021increasing}. 
\begin{theorem}
\label{theorem: constant}
Let $\y = (c^n) \in\N^n$ and $\x = (x_1, x_2, \ldots, x_n) \in [\sum_{i=1}^n y_i]^n$.
Then, $\x \in \PAINV_{n}(\y)$ if and only if
\begin{enumerate}[label=(\arabic*)]
    \item\label{it: constant (1)} $x_{i} \equiv 1 \mod c$, for all $i \in [n]$, and
    \item\label{it: constant (2)} $|\{i \in [n]:x_i \leq c \cdot j\}| \geq j$, for all $j \in [n]$.
\end{enumerate}
\end{theorem}
\begin{proof}
We first claim that $\x \in \PAINV_n(\y)$ implies both \ref{it: constant (1)} and \ref{it: constant (2)} hold.
We prove this by contrapositive, which is to say that if either of \ref{it: constant (1)} or \ref{it: constant (2)} does not hold, then $\x \notin \PAINV_n(\y)$.
Suppose \ref{it: constant (1)} does not hold, meaning there exists $i \in [n]$ with $x_i \not\equiv 1 \mod c$.
Consider the permutation $\x'$ of $\x$ given by $\x' = (x_i, x_1, \ldots, x_{i-1}, x_{i+1}, \ldots, x_n)$. 
After the first car parks in spot $x_i$, the number of unoccupied spots to its left is not a multiple of $c$, and so no subset of subsequent cars can fully occupy them.
This shows that $\x\notin \PAINV_n(\y)$ in this first case.
Now suppose \ref{it: constant (2)} does not hold, meaning there exists $j \in [n]$ with $|\{i \in [n] : x_i \leq c \cdot j\}| < j$.
This implies $|\{i \in [n] : x_i > c \cdot j\}| > n - j$, which is to say that more than $n - j$ preferences are greater than $c  j$.
It follows that parking will fail under any permutation $\x'$ of $\x$ such that the first car has one of these preferences greater than $c j$.
However, there are only $c  n - c  j= c(n-j)$ spots to the right of spot $c  j$, meaning at most $n - j$ cars can park there. This shows that $\x\notin \PAINV_n(\y)$ in this second case.

Next, we show that if both \ref{it: constant (1)} and \ref{it: constant (2)} hold, then $\x \in \PAINV_n(\y)$.
Note that \ref{it: constant (1)} implies cars can only park with their rear bumpers in one of the $n$ spots that are congruent to $1 \mod c$. Otherwise, upon the arrival of a car, it either prefers and parks in a spot that is not congruent to $1 \mod c$ or parks immediately after a sequence of consecutively parked cars, the first of which must have preferred and parked in a spot that is not congruent to $1 \mod c$.
We now show via contrapositive that if \ref{it: constant (1)} holds, then \ref{it: constant (2)} implies $\x\in\PAINV_n(\y)$.
Suppose $\x \notin \PA_n(\y)$ (recall $\PAINV_n(\y) \subseteq \PA_n(\y)$).
Then, for some $j \in [n]$, more than $n - j$ cars want to park in the $c(n - j)$ spots to the right of spot $c  j$.
This implies $|\{i \in [n]:x_i > c \cdot j\}| > n - j$ and therefore $|\{i \in [n]:x_i \leq c \cdot j\}| < j$ for some $j \in [n]$, violating \ref{it: constant (2)}. 
This shows that \ref{it: constant (1)} and \ref{it: constant (2)} together imply $\x\in\PA_n(\y)$.
Lastly, note that \ref{it: constant (1)} and \ref{it: constant (2)} are preserved under permutations, so $\x \in \PAINV_n(\y)$.
\end{proof}

Given that the characterization provided in Theorem~\ref{theorem: constant} is identical to that of parking sequences for constant car lengths (\cite[Theorem 3.4]{adeniran2021increasing}), the following is immediate. 
\begin{corollary}
\label{proposition: same}
If $\y=(c^n)$, then $\PSINV_n(\y)=\PAINV_n(\y)$.
\end{corollary}
Note that Corollary~\ref{proposition: same} does not hold without the invariant condition. For example, if $\y=(2,2,2)$, then $\x=(3,2,1)$ is a parking assortment, yet it is not a parking sequence.  

Adeniran and Yan give a determinant formula for the number of nondecreasing parking sequences of length $n$ \cite[Corollary 2.4]{adeniran2021increasing}.
They also show that the number of nondecreasing parking sequences of length $n$ with constant car lengths $c\in\N$ 
is given by the Fuss-Catalan 
numbers\footnote{OEIS  \href{http://oeis.org/A355172}{A355172}.}, $A_{k,n}=\frac{1}{kn+1}\binom{kn + n}{n}$ \cite[Corollary 2.6]{adeniran2021increasing}.
They also show that the number of invariant parking sequences of length $n$ with constant car lengths $c \in \N$ is given by $(n+1)^{n-1}$, independent of $c$ \cite[Corollary 3.5]{adeniran2021increasing}.
We establish that the number of nondecreasing invariant parking assortments of length $n$ with constant car lengths $c \in \N$ is given by $C_n=\frac{1}{n+1}\binom{2n}{n}$, the $n$th Catalan number\footnote{OEIS \href{http://oeis.org/A000108}{A000108}.}, independent of $c$. 
Note that by Corollary \ref{proposition: same}, this result also shows the Catalan numbers enumerate nondecreasing invariant parking sequences with constant car lengths. 

In what follows, we use \[\PAINVND_n(\y) := \{\x = (x_1, x_2, \ldots, x_n) \in \PAINV_n(\y) : x_1 \leq x_2 \leq \cdots \leq x_n \}\] to denote the set of nondecreasing invariant parking assortments given $\y \in \N^n$.

\begin{corollary}
\label{corollary: catalan count}
If $\y = (c^n) \in \N^n$, then $\left|\PAINVND_n(\y)\right| = C_n$, where $C_n$ is the $n$th Catalan number.
\end{corollary}
\begin{proof}
We establish a bijection between $\PAINVND_n(\y)$ and the set of nondecreasing parking functions of length $n$, which we denote by $\PFND_n$.
The latter is well-known to be enumerated by $C_n$, and so constructing a bijection
$\varphi: \PAINVND_n(\y) \rightarrow \PFND_n$ 
is sufficient to prove the result.

Given $\x=(x_1,\ldots,x_n) \in \PAINVND_n(\y)$, define $\varphi: \PAINVND_n(\y) \rightarrow \PFND_n$
by
\begin{align*}
    \varphi(\x) = \left(1 + \frac{x_1-1}{c}, 1 + \frac{x_2-1}{c}, \ldots,1 + \frac{x_n-1}{c}\right).
\end{align*} 
Since $\x$ is nondecreasing, $\varphi(\x)$ is nondecreasing. 
Also, since $\x$ is nondecreasing, conditions \ref{it: constant (1)} and \ref{it: constant (2)} of Theorem \ref{theorem: constant} together are equivalent to $x_i \in \{1, c + 1, \ldots, (i - 1) c + 1\}$ for all $i \in [n]$. It follows that $\varphi(\x)_i \in [i]$ for all $i \in [n]$, where $\varphi(\x)_i$ denotes the $i$th entry of $\varphi(\x)$.
This shows that $\varphi(\x) \in \PFND_n$.

We establish that $\varphi$ is a bijection by constructing its inverse.
Given $\z=(z_1, z_2, \ldots, z_n) \in \PFND_n$, define $\psi: \PFND_n \rightarrow \PAINVND_n(\y)$
\begin{align*}
    \psi(\z)= \left(1 + (z_1-1)c, 1 + (z_2-1)c, \ldots,1 + (z_n-1)c \right).
\end{align*}
Since $\z$ is nondecreasing, $\psi(\z)$ is nondecreasing.
Also, $\z \in \PFND_n$ implies $z_i \in [i]$ for all $i \in [n]$.
In turn this implies $\psi(\z)_i \in \{1,c+1,\ldots,(i-1)c+1\}$ for all $i \in [n]$, where $\psi(\z)_i$ is the $i$th entry of $\psi(\z)$.
Since $\psi(\z)$ is nondecreasing, $\psi(\z)_i \in \{1,c+1,\ldots,(i-1)c+1\}$ for all $i \in [n]$ is equivalent to conditions \ref{it: constant (1)} and \ref{it: constant (2)} of Theorem \ref{theorem: constant}.
This shows that $\psi(\z) \in \PAINVND_n(\y)$. 

Lastly, observe that $(\psi \circ \varphi)(\x) = \x$ and $(\varphi \circ \psi)(\z) = \z$, meaning $\psi = \varphi^{-1}$.
\end{proof}

The next result follows immediately from Corollary~\ref{proposition: same} and \cite[Corollary 3.5]{adeniran2021increasing}. We give an independent proof utilizing Corollary \ref{corollary: catalan count}.
\begin{corollary}
\label{corollary: constant count}
If $\y = (c^n) \in \N^n$, then $|\PAINV_n(\y)| = (n+1)^{n-1}$.
\end{corollary}
\begin{proof}
Note that the bijection $\varphi: \PAINVND_n(\y) \rightarrow \PFND_n$ in the proof of Corollary~\ref{corollary: catalan count} preserves the number of distinct entries, i.e., $|\{x_1, x_2, \ldots, x_n\}| = |\{\varphi(\x)_1, \varphi(\x)_2, \ldots, \varphi(\x)_n\}|$ for every $\x \in \PAINVND(\y)$.
Therefore, the number of distinct rearrangements of $\x \in \PAINVND(\y)$ is the same as the number of distinct rearrangements of $\varphi(\x) \in \PFND_n$.
Lastly, note that $\PF_n$ consists of the union of the sets of rearrangements of the elements of $\PFND_n$.
\end{proof}

\section{Minimally Invariant Car Lengths}
\label{section: minimally invariant car lengths}
Given the technicality of the results in the previous section, we now introduce the concept of minimally invariant car lengths. 
We remark that since all parking sequences are parking  assortments, our results for minimally invariant car lengths also hold in the context of parking sequences. 

\begin{definition}
We say $\y= (y_1, y_2, \ldots, y_n) \in \N^n$ is \emph{minimally invariant} if $\PAINV_n(\y) = \{(1^n)\}$.
\end{definition}

This definition was motivated by the following result of Adeniran and Yan~\cite{adeniran2021increasing}.
\begin{theorem}[\cite{adeniran2021increasing}, Theorem~3.2.]
\label{theorem: increasing}
Let $\y = (y_1, y_2, \ldots,y_n) \in \N^n$ be strictly increasing.
Then, $\PSINV_n(\y) = \{(1^n)\}$.
\end{theorem}

We begin by noting that if $\PAINV_n(\y)=\{(1^n)\}$, then  $\PSINV_n(\y)=\{(1^n)\}$ as well. 
However, the converse is not true. 
For example, if $\y=(1,2,1)$, then  $\PSINV_3(\y)=\{(1,1,1)\}$, while 
\begin{align*}
    \PAINV_3(\y)=\{(1,1,1),(1,1,2),(1,2,1),(2,1,1),(1,1,3),(1,3,1),(3,1,1)\}.
\end{align*}
Moreover, we remark that while sufficient, $\y$ being strictly increasing is not a necessary condition for $\PAINV_n(\y) = \{(1^n)\}$.
For example, $\y = (1, 2, 2)$ is weakly increasing yet it satisfies $\PAINV_3(\y) = \{(1,1,1)\}$.
Hence, we state and prove a more general result (Theorem~\ref{theorem: mi}) which fully characterizes when $\y$ is minimally invariant. 
As a consequence of this result, Theorem~\ref{theorem: increasing length} yields the analogous result to Theorem~\ref{theorem: increasing} for parking assortments.

Next, we use Lemma~\ref{lemma: removal} to establish that if $\y=(y_1, y_2, \ldots,y_n) \in \N^n$ is minimally invariant and $y_{n + 1} \in \N$, then every  $\x\in\PAINV_{n + 1}((y_1, y_2, \ldots, y_n, y_{n+1}))$ 
has $n$ ones and one entry $w \in \N$.
This technical result is stated below and is instrumental in proving Theorem \ref{theorem: mi}.

\begin{corollary}
\label{corollary: append}
Let $\y = (y_1, y_2, \ldots, y_n) \in \N^n$, $y_{n + 1} \in \N$, and $\z = (y_1, y_2, \ldots, y_n, y_{n+1})$.
If $\y$ is minimally invariant, then every nondecreasing $\x \in \PAINV_{n+1}(\z)$ is of the form $\x = (1^n, w)$ for some $w \in \N$.
\end{corollary}
\begin{proof}
Let $\x = (x_1, x_2, \ldots, x_n, x_{n+1}) \in\PAINV_{n+1}(\z)$ be nondecreasing, so $n+1 = \argmax_{i \in [n+1]} x_i$.
Note that $\y = \z_{\vert_{n}}$ by construction.
Then, by Lemma~\ref{lemma: removal}, $\x_{\widehat{n+1}} = (x_1, x_2, \ldots, x_n) \in \PAINV_{n}(\y)$.
Moreover, since $\y$ is minimally invariant, it must be that $\x_{\widehat{n+1}}  = (1^n)$. 
This can only happen if $\x$ is of the form $\x=(1^n,w)$ for some $w \in \N$. 
In particular, $w=x_{n+1}$.
\end{proof}

The following result further extends Corollary~\ref{corollary: append} by characterizing the possible values for $w \in \N$.
Although its statement is technical, the proof makes use of the following observation.
An invariant parking assortment with entries of all ones and a single $w$ implies that the cars in the queue before the car with preference $w$ park in sequential order beginning at the start of the street. Then the following two things can occur:
(1) the car with preference $w$ parks immediately after the previous cars, or (2) the car with preference $w$ parks leaving a gap between its position and that of the previous car. 
In the second case, we are ensured that there are some cars remaining (of appropriately small lengths) in order to fill in the gap of spaces left unoccupied by the car with preference $w$, as otherwise the preference list would not be a parking assortment, let alone invariant.  

\begin{corollary}
\label{corollary: gap}
Let $\y$, $\z$, and $\x = (1^n, w)$ be as in Corollary~\ref{corollary: append}.
For any such $w$, the following holds for all $i \in [n]$:
\begin{enumerate}[label=(\arabic*)]
    \item\label{it: gap (1)} $w \leq 1 + \sum_{j=1}^{i-1} y_j$, or
    \item there exists an increasing sequence $s_1, s_2, \ldots, s_m \in \{i+1,i+2,\ldots,n\}$ with
    \begin{enumerate}[label=(2\alph*)]
        \item\label{it: gap (2a)} $\sum_{\ell=1}^m y_{s_\ell}=(w - 1)-\sum_{j=1}^{i-1} y_j$, and
        \item\label{it: gap (2b)} $y_k > (w - 1)-\sum_{j=1}^{i-1} y_j - \sum_{j = 1}^\ell y_{s_j}$ for all $\ell \in [m]$ and $k \in \{s_{\ell - 1} + 1, s_{\ell - 1} + 2, \ldots, s_\ell - 1\}$, where $s_0 = i$.
    \end{enumerate}
\end{enumerate}
\end{corollary}

\begin{proof}
We establish conditions on $w \in \N$ under which $\x =(1^n,w) \in \PAINV_{n+1}(\z)$.
For any such $w$, parking succeeds on every permutation of $\x = (1^n, w)$, which is to say parking succeeds regardless of which car has preference $w$.
Suppose car $i \in [n]$ has preference $w$.
Then, cars $1, 2, \ldots, i -1$ all prefer spot $1$ and park consecutively in the first $\sum_{l=1}^{i-1}y_i$ spots.

If condition \ref{it: gap (1)} holds, car $i$ parks immediately after car $i - 1$.
Similarly, cars $i + 1, i + 2, \ldots, n$, all of which prefer spot $1$, park immediately and consecutively after car $i$.
This shows parking succeeds whenever condition \ref{it: gap (1)} holds.
Suppose condition \ref{it: gap (1)} does not hold, which is to say $w > 1 + \sum_{j=1}^{i-1} y_j$.
In this case, after cars $1, 2, \ldots, i$ park, there is a gap (of unoccupied spots) of size $(w-1)-\sum_{j=1}^{i-1} y_j$ between the rear of car $i$ and the front of car $i - 1$.
Parking succeeds if and only if this gap is filled.
Note that this happens if and only if there is a sequence $s_1, s_2, \ldots, s_m \in \{i + 1, i + 2, \ldots, n\}$ of cars arriving after car $i$ whose sum of lengths fills the gap exactly (this is condition \ref{it: gap (2a)}), and such that there is no car outside the sequence that interferes with its consecutive parking.
The latter is to say that, for every $\ell \in [m]$, all cars $k \in \{s_{\ell - 1} + 1, s_{\ell - 1} + 2, \ldots, s_\ell - 1\}$ that arrive after car $s_{\ell - 1}$ but before car $s_\ell$, which we note are not part of the sequence, have length $y_k$ bigger than the remaining gap size $(w - 1)-\sum_{j=1}^{i-1} y_j - \sum_{j = 1}^\ell y_{s_j}$ at their time of arrival.
This is condition \ref{it: gap (2b)}.
\end{proof}

We now state and prove our main result.
\begin{theorem}
\label{theorem: mi}
Let $\y = (y_1, y_2, \ldots, y_n) \in \N^n$. Then, $\y$ is minimally invariant if and only if there does not exist $w \in \N_{> 1}$ such that $(1^{n-1},w) \in \PAINV_n(\y)$.
\end{theorem}
\begin{proof}
Note that if $\y$ is minimally invariant, then by definition, $\PAINV_n(\y) = \{(1^n)\}$, meaning there does not exist $w \in \N_{> 1}$ such that $(1^{n-1},w) \in \PAINV_n(\y)$.

We now prove that if there does not exist $w \in \N_{> 1}$ such that $(1^{n-1},w) \in \PAINV_n(\y)$, then $\y$ is minimally invariant. We prove the contrapositive: if $\y$ is not minimally invariant, then there exists $w\in \N_{>1}$ such that $(1^{n-1},w) \in \PAINV_n(\y)$.
We proceed by induction on $n$. 
If $n=1$, the statement holds vacuously as the only possible parking assortment is $(1)$, which implies that $\y$ is minimally invariant and that the assumption is false.
If $n=2$, in Theorem~\ref{theorem: invariant pair} we show from first principles that $\y=(y_1,y_2)$ is not minimally invariant if and only if $y_1 \geq y_2$. 
In particular, we show that $(1,1 + y_2) \in \PAINV_2(\y)$ whenever $y_1 \geq y_2$, so we may take $w = 1 + y_2 > 1$.

Now, assume the statement holds when $n = k$ for some $k \in \N$.
We need to show that the statement holds when $n = k + 1$, which is to say that if $\y=(y_1,y_2,\ldots,y_{k+1}) \in \N^{k+1}$ is not minimally invariant, then there exists $w \in \N_{>1}$ such that $(1^k,w) \in \PAINV_{k+1}(\y)$.
Let $\y_{\vert_{k}} = (y_1, y_2, \ldots, y_k)$.
There are two mutually exclusive possibilities:
\begin{itemize}
    \item[]
    Case 1:
    Suppose $\y_{\vert_{k}}$ is minimally invariant.
    By Corollary~\ref{corollary: append}, every nondecreasing $\x \in \PAINV_{k+1}(\y)$ is of the form $(1^k, w)$ for some $w \in \N$.
    However, by assumption, $\y$ is not minimally invariant, and so there must exist such a $w$ with $w > 1$.
    \item[] 
    Case 2:
    Suppose $\y_{\vert_{k}}$ is not minimally invariant.
    By the inductive hypothesis, there exists $w \in \N_{> 1}$ such that $(1^{k-1}, w) \in \PAINV_{k}(\y_{\vert_k})$.
    Then, 
    \begin{align*}
        w \leq 1 + \sum_{i=1}^{k-1} y_i < 1 + \sum_{i=1}^{k} y_i,
    \end{align*}
    where the first inequality follows from $\PAINV_{k}(\y_{\vert_k}) \subseteq \PA_{k}(\y_{\vert_k})$, Proposition~\ref{proposition: nondecreasing} and the fact that $(1^{k-1}, w)$ is nondecreasing.
    The second inequality, together with Proposition~\ref{proposition: nondecreasing}, further implies $(1^{k}, w) \in \PA_{k+1}(\y)$.
    Lastly, note that any non-trivial permutation of $(1^k, w)$ places $w$ among the first $k$ entries.
    Since $(1^{k-1}, w) \in \PAINV_{k}(\y_{\vert_{k}})$, under any such permutation, the first $k$ cars occupy the first $\sum_{i=1}^k y_i$ spots, leaving the last spots numbered $1 + \sum_{i=1}^k y_i, \ldots, \sum_{i=1}^{k+1} y_i$ unoccupied for car $k+1$ (which prefers spot $1$ and has length $y_{k+1}$) to fill.
\end{itemize}
This shows $(1^k, w) \in \PAINV_{k+1}(\y)$ for some $w \in \N_{> 1}$ in either case.
\end{proof}
\begin{remark}
Theorem~\ref{theorem: mi} is a concise characterization for $\y \in \N^n$ to be minimally invariant in the sense that there are only $n$ distinct permutations of $(1^{n-1}, w)$ and, by Proposition~\ref{proposition: nondecreasing}, we only need to check $w \in \N$ in the range $1 < w \leq 1 + \sum_{j=1}^{i-1} y_j$.
It is a pseudopolynomial-time characterization since the number of parking experiment calls depends on $\sum_{i=1}^n y_i$.
\end{remark}

As a first corollary to Theorem~\ref{theorem: mi}, we find that if a car length list is minimally invariant, then so is its restriction to the first $i \in [n]$ entries.
\begin{corollary}
\label{corollary: restriction}
If $\y \in \N^n$ is minimally invariant, then $\y_{\vert_i}$ is minimally invariant for all $i \in [n]$.
\end{corollary}
\begin{proof}
If $i = n$, the statement holds since $\y|_n=\y$ and since by assumption $\y$ is minimally invariant. 
We prove the contrapositive for $i=n-1$. 
That is, we show that if $\y_{\vert_{n-1}}$ is not minimally invariant, then $\y$ is not minimally invariant. 
Suppose $\y_{\vert_{n-1}}$ is not minimally invariant.
Then, by Theorem~\ref{theorem: mi}, there exists $w \in \N_{> 1}$ such that $(1^{n-2},w)\in\PAINV_{n-1}(\y_{\vert_{n-1}})$. 
We claim that $\x = (1^{n-1},w) \in \PAINV_n(\y)$.
To see this, consider the possible permutations $\x'$ of $\x$:
\begin{itemize}
    \item[]
    Case 1:
    Suppose $\x' = (1^{n-1}, w)$. 
    Since $(1^{n-2},w)\in\PAINV_{n-1}(\y_{\vert_{n-1}}) \subseteq \PA_{n-1}(\y_{\vert_{n-1}})$ is nondecreasing, Proposition~\ref{proposition: nondecreasing} implies that
    \begin{align*}
        w \leq 1 + \sum_{i=1}^{n-2} y_i < 1 + \sum_{i=1}^{n-1} y_i.
    \end{align*}
    The second inequality, together with Proposition~\ref{proposition: nondecreasing}, further implies $\x' \in \PA_{n}(\y)$.
    \item[]
    Case 2:
    Suppose $\x' \neq (1^{n-1}, w)$.
    Then, $w$ lies among the first $n-1$ entries of $\x'$.
    Since $(1^{n-2}, w) \in \PAINV_{n-1}(\y_{\vert_{n-1}})$, the first $n - 1$ cars under $\x'$ occupy the first $\sum_{i=1}^{n-1} y_i$ spots, leaving spots
    \begin{align*}
        1 + \sum_{i=1}^{n-1} y_i, \ldots, \sum_{i=1}^{n} y_i
    \end{align*}
    unoccupied for car $n$ (which prefers spot $1$ and has length $y_n$) to fill.
    This shows $\x' \in \PA_{n}(\y)$ in this case.
\end{itemize}
These cases establish that $\x = (1^{n-1},w) \in \PAINV_n(\y)$, which implies $\y$ is not minimally invariant.
The claim for $i \in [n- 2]$ follows by applying the argument inductively. 
In particular, if we assume that $\y_{\vert_{i}}$ is minimally invariant for some $i \in \{2, 3, \ldots, n\}$, the same argument implies that $\y_{\vert_{i - 1}}$ is minimally invariant.
\end{proof}

The next result unfolds what must take place in the parking experiments specified in Theorem~\ref{theorem: mi} so that the list of cars length is minimally invariant. 
\begin{corollary}
\label{corollary: mi (alternate)}
Let $\y = (y_1, y_2, \ldots, y_n) \in \N^n$. Then, $\y$ is minimally invariant if and only if
for every $w \in \{2, 3, \ldots, \sum_{i=1}^n y_i\}$, there exists $i \in [n]$ such that both of the following hold:
\begin{enumerate}[label=(\arabic*)]
    \item\label{it: mi (alternate) (1)} $w > 1 + \sum_{j=1}^{i-1} y_j$, and
    \item\label{it: mi (alternate) (2)} there does not exist an increasing sequence $s_1, s_2, \ldots, s_m \in \{i+1,i+2,\ldots,n\}$ with
    \begin{enumerate}[label=(2\alph*)]
        \item\label{it: mi (alternate) (2a)} $\sum_{\ell=1}^m y_{s_\ell}=(w - 1)-\sum_{j=1}^{i-1} y_j$, and
        \item\label{it: mi (alternate) (2b)} $y_k > (w - 1)-\sum_{j=1}^{i-1} y_j - \sum_{j = 1}^\ell y_{s_j}$ for all $\ell \in [m]$ and $k \in \{s_{\ell - 1} + 1, s_{\ell - 1} + 2, \ldots, s_\ell - 1\}$, where $s_0 = i$.
    \end{enumerate}
\end{enumerate}
\end{corollary}

\begin{proof}
By Theorem~\ref{theorem: mi}, $\y$ is minimally invariant if and only if $(1^{n-1},w)\not\in\PAINV_n(\y)$ for all $w\neq 1.$
Equivalently, $\y$ is minimally invariant if and only if $(1^{n-1},w)\not\in\PAINV_n(\y)$ for all $w\in\{2,3,\ldots,\sum_{i=1}^n y_i\}$ (parking always fails when $w>\sum_{i=1}^n y_i$, since these choices of $w$ are preferences outside the parking lot). By Corollary~\ref{corollary: gap}, this is equivalent to saying that for all $w\in\{2,3,\ldots,\sum_{i=1}^n y_i\},$ there exists $i\in[n]$ such that \ref{it: mi (alternate) (1)} and \ref{it: mi (alternate) (2)} hold.
\end{proof}

As another application of our characterization, we provide an analogous
result to Theorem~\ref{theorem: increasing} for parking assortments. 
\begin{theorem}
\label{theorem: increasing length}
Let $\y = (y_1, y_2, \ldots,y_n) \in \N^n$ be strictly increasing.
Then, $\PAINV_n(\y) = \{(1^n)\}$.
\end{theorem}

\begin{proof}
We prove this by induction on $n$.
If $n=1$, the claim holds trivially. 
Now, assume the statement holds when $n = k$ for some $k \in \N$.
We need to show that the statement holds when $n = k + 1$.
Let $\y = (y_1, y_2, \ldots, y_k, y_{k+1}) \in \N^{k+1}$ be strictly increasing.
Note that $\y_{\vert_{k}} = (y_1, y_2, \ldots, y_k)$ is also strictly increasing and therefore, by the inductive hypothesis, minimally invariant.
Then, by Corollary~\ref{corollary: append}, all nondecreasing $\x \in \PAINV_{k+1}(\y)$ are of the form $\x = (1^k, w)$ for some $w \in \N$. 

Suppose, by way of contradiction, that $\x = (1^k, w) \in \PAINV_{k+1}(\y)$ with $w \in \N_{> 1}$.
First, note that since $\PAINV_{k+1}(\y) \subseteq \PA_{k+1}(\y)$ and $\x$ is nondecreasing, Proposition~\ref{proposition: nondecreasing} implies
\begin{align}
    w \leq 1 + \sum_{j=1}^{k} y_j.\label{eq:to contradict}
\end{align}
In what follows, we will show that there is a contradiction arising due to the inequality in \eqref{eq:to contradict}.

By the inductive hypothesis we know that $\y_{\vert_k}$ is minimally invariant, hence for some $i \in [k]$, $\x'=(1^{i-1},w,1^{k-i})$ fails to park the cars given $\y_{\vert_k}$.
Now let $\x''=(1^{i-1},w,1^{k-i+1})$ be $\x'$ with an additional $1$ added at the end. 
Note that this is a rearrangement of the original preference list $\x=(1^k,w)$.
Moreover, note that, under $\x''$, there is 
a subset of the first $\sum_{i=1}^k y_i$ spots that remains unoccupied after the arrival of car $k$. 

If parking fails under $\x''$, we are done.
Now assume parking succeeds under $\x''$.
We will shortly show that this can only happen if $w=(\sum_{j=1}^{k+1} y_j) - y_i + 1$.
This then yields a contradiction to inequality \eqref{eq:to contradict} since 
\[w=\left(\sum_{j=1}^{k+1} y_j\right) - y_i + 1 \nleq 1 + \sum_{j=1}^{k} y_j\]
with $i \in [k]$ and $\y$ strictly increasing. 

Let us now prove our claim that $w=(\sum_{j=1}^{k+1} y_j) - y_i + 1$. 
First, suppose $w>(\sum_{j=1}^{k+1} y_j) - y_i+1.$ Then parking fails under $(1^k,w),$ since the number of spots to the inclusive right of $w$ is $(\sum_{j=1}^{k+1} y_j) - w + 1 < (\sum_{j=1}^{k+1} y_j) - (\sum_{j=1}^{k+1} y_j - y_i+1) + 1=y_i<y_{k+1},$ so car $k+1$ is unable to park.

Now, suppose that $w<(\sum_{j=1}^{k+1} y_j) - y_i+1$. 
Then after car $i$ with preference $w$ parks, it leaves a gap (of consecutive unoccupied spots) to its left and another to its right. 
We now show that the left gap will be completely filled after cars $1,2,\ldots,k$ have parked. 
Suppose car $j\in\{i+1, i + 2, \ldots, k\}$ parks in the right gap, and let $g$ be the size of the left gap at this point in the parking experiment. 
If $g=0,$ then the gap has been filled, so assume $g>0$.
Since car $j$ has preference 1, it tried and failed to park in the left gap, meaning $y_j>g$.
But since $\y$ is strictly increasing, the remaining cars will also be too large to park in the left gap, so parking will fail, a contradiction. 

So some subset of the cars $i+1, i + 2, \ldots,k$ fills the left gap, and the rest park at the beginning of the right gap consecutively (since they have preference 1), leaving the last $y_{k+1}$ spaces empty. 
But this contradicts that the first $k$ cars leave a gap in spaces $1, 2, \ldots, \sum_{j=1}^k y_j$. 

Thus we must have $w=(\sum_{j=1}^{k+1}y_j)-y_i+1.$ 
This choice of $w$ gives a contradiction, so $(1^k,w)\not\in\PAINV_{k+1}(\y)$. This shows that $(1^k,w)\not\in\PAINV_{k+1}(\y)$ for all $w\in\mathbb N_{>1}$. Thus, $\y$ is minimally invariant by Theorem~\ref{theorem: mi}. This completes the inductive step and proves the claim.
\end{proof}

\section{Two and Three Cars}
\label{section: two and three cars}

\subsection{Two Cars}
We now restrict $\y = (y_1, y_2) \in\N^2$. 
Our main result gives a characterization of invariant parking assortments with two cars (Theorem~\ref{theorem: invariant pair}).
As an immediate corollary (Corollary~\ref{corollary: mi pair}), we obtain a full characterization for minimally invariant parking assortments with two cars.
In addition, we provide an alternate proof of the Corollary~\ref{corollary: mi pair} using Theorem~\ref{theorem: mi}.
\begin{theorem}
\label{theorem: invariant pair}
Let $\y=(y_1, y_2)\in\N^2$.
\begin{enumerate}
    \item If $y_1 < y_2$, then $\PAINV_2(\y) = \{(1,1)\}$.
    \item If $y_1 \geq y_2$, then $\PAINV_2(\y) = \{(1,1), (1, y_2 + 1), (y_2 + 1, 1)\}$.
\end{enumerate}
\end{theorem}
\begin{proof}[Proof]
Clearly, $\{(1,1)\} \subseteq \PAINV_2(\y)$ in either case.

First, consider the case in which $y_1 < y_2$.
Suppose $\x = (x_1, x_2) \in [y_1 + y_2]^2$ satisfies $\x\in \PAINV_2(y)$ and $\x \neq (1, 1)$. Let $\x' = (x_1', x_2')$ be the nondecreasing rearrangement of $\x$.
Note that $\x \in \PAINV_2(\y)$ implies $\x' \in \PAINV_2(\y)$ and recall that $\PAINV_2(\y) \subseteq \PA_2(\y)$.
Therefore, by Proposition~\ref{proposition: nondecreasing}, we have $x_1' = 1$ and $1 \leq x_2' \leq y_1 + 1$.
Since $x_1' = 1$ and $\x\neq(1, 1)$, we in fact have $1 < x_2' \leq y_1 + 1$.
Now, consider the rearrangement $\x'' = (x_2', x_1') = (x_2', 1)$.
Under $\x''$, car $1$ parks in spots $x_2', x_2' + 1, \ldots, x_2' + y_1 - 1$, leaving $x_2' - 1$ spots available to the left of spot $x_2'$ and $y_2 - x_2' + 1$ spots available to the right of spot $x_2' + y_1 - 1$.
Note that $x_2' - 1 \leq y_1 + 1 - 1 = y_1 < y_2$, and so car $2$ under $\x''$ is unable to park to the left of spot $x_2'$.
Similarly, $y_2 - x_2' + 1< y_2 - 1 + 1 = y_2$, and so car $2$ under $\x''$ is unable to park to the right of spot $x_2' + y_1 - 1$.
This shows that $\x'' \notin \PA_2(\y)$, which implies $\x \notin \PAINV_2(\y)$, a contradiction.

Next, consider the case in which $y_1 \geq y_2$. 
We first show that $\{(1, 1), (1, y_2 + 1), (y_2 + 1, 1)\} \subseteq \PAINV_2(\y)$.
To see that $(1, y_2 + 1) \in \PA_2(\y)$, note that upon the arrival of car $1$, it parks in spots $1, 2, \ldots, y_1$, leaving $y_2$ available spots to the right of spot $y_1$.
If $y_2 < y_1$, upon the arrival of car $2$, it finds spot $y_2 + 1$ occupied and continues moving forward until it finds and parks in the $y_2$ spots $y_1 + 1, y_1 + 2, \ldots, y_1 + y_2$ available to the right of spot $y_1$.
If $y_2 = y_1$, it similarly parks in spots $y_1 + 1, y_1 + 2, \ldots, y_1 + y_2$.
To see that $(y_2 + 1, 1) \in \PA_2(\y)$, note that upon the arrival of car $1$, it parks in spots $y_2 + 1, y_2 + 2, \ldots, y_2 + y_1$, leaving $y_2$ available spots to the left of spot $y_2 + 1$.
Upon the arrival of car $1$, it finds and parks in the $y_2$ spots $1, 2, \ldots, y_2$ available to the left of spot $y_2 + 1$.
This shows that $\{(1, 1), (1, y_2 + 1), (y_2 + 1, 1)\} \subseteq \PAINV_2(\y)$.
Conversely, suppose $\x = (x_1, x_2) \in [y_1 + y_2]^2$ with $\x \notin\{ (1, 1), (1, y_2 + 1), (y_2 + 1,1)\}$ satisfies $\x \in \PAINV_2(\y)$ and let $\x' = (x_1', x_2')$ be its nondecreasing rearrangement.
Note that $\x \in \PAINV_2(\y)$ implies $\x' \in \PAINV_2(\y)$ and recall that $\PAINV_2(\y) \subseteq \PA_2(\y)$.
Therefore, by Proposition~\ref{proposition: nondecreasing}, we have $x_1' = 1$ and $1 \leq x_2' \leq y_1 + 1$.
Since $x_1' = 1$ and $\x \notin\{ (1, 1), (1, y_2 + 1), (y_2 + 1,1)\}$, we in fact have $x_2' \neq 1, y_2 + 1$.
This means that either $1 < x_2' < y_2 + 1$ or $y_2 + 1 < x_2' \leq y_1 + 1$.
Consider the rearrangement $\x'' = (x_2', x_1') = (x_2', 1)$ and the case in which $1 < x_2' < y_2 + 1$.
Upon the arrival of car $1$ under $\x''$, it parks in spots $x_2', x_2' + 1, \ldots, x_2' + y_1 - 1$, leaving $x_2' - 1 < y_2 + 1 - 1 = y_2$ spots to the left of spot $x_2'$ and $y_2 - x_2' + 1 < y_2$ spots to the right of spot $x_2' + y_1 - 1$.
Therefore, car $2$ of length $y_2$ under $\x''$ is unable to park.
Similarly, consider the case in which $y_2 + 1 < x_2' \leq y_1 + 1$.
Car $1$ of length $y_1$ under $\x''$ finds that there are only $y_1 + y_2 - x_2' + 1 < y_1 + y_2 - y_2 - 1 = y_1$ spots available until the end of the road, and so it is unable to park.
This shows that $\x'' \notin \PA_2(\y)$ in either case, which implies $\x \notin \PAINV_2(\y)$, a contradiction.
\end{proof}

\begin{corollary}
\label{corollary: mi pair}
Let $\y=(y_1, y_2)\in\N^2$.
Then, $\y$ is minimally invariant if and only if $y_1 < y_2$.
\end{corollary}
\begin{proof}[Alternate proof of Corollary~\ref{corollary: mi pair}]
By Theorem \ref{theorem: mi}, $\y$ is minimally invariant if and only if there does not exist $w \in \N_{> 1}$ satisfying $(1,w), (w,1) \in \PA_2(\y)$.
Note that $(1,w) \in \PA_2(\y)$ if and only if $w \leq 1+y_1 $, whereas $(w, 1) \in \PA_2(\y)$ if and only if $w = 1 + y_2$. 
Therefore, $\y$ is minimally invariant if and only if there does not exist $w \in \N_{> 1}$ satisfying $w \leq 1 + y_1$ and $w = 1 + y_2$.

Suppose $y_1 < y_2$. 
Then, there does not exist $w \in \N_{> 1}$ satisfying $w \leq 1 + y_1$ and $w = 1 + y_2$, for otherwise $1 + y_2 = w \leq 1 + y_1$, implying $y_2 \leq y_1$ and contradicting $y_1 < y_2$.
This shows that if $y_1 < y_2$, then $\y$ is minimally invariant.
Conversely, suppose $y_1 \geq y_2$.
Then, letting $w = 1 + y_2$ satisfies $w > 1$, $w \leq 1 + y_1$, and $w = 1 + y_2$, meaning $\y$ is not minimally invariant.
Equivalently, this shows that if $\y$ is minimally invariant, then $y_1 < y_2$.
\end{proof}

\subsection{Three Cars}
We now characterize minimally invariant parking assortments with three cars (Theorem~\ref{theorem: invariant pair}).
The proof extends the style of the alternate proof of Corollary~\ref{corollary: mi pair} to the three-car setting.
\begin{theorem}
\label{theorem: mi triple}
Let $\y=(y_1,y_2,y_3)\in\N^3$.
Then, $\y$ is minimally invariant if and only if $y_1<y_2$, $y_1 < y_3$, and $y_1+y_3 \neq y_2$. 
\end{theorem}
\begin{proof}
By Theorem \ref{theorem: mi}, $\y$ is minimally invariant if and only if there does not exist $w \in \N_{> 1}$ satisfying $(1,1,w), (1,w,1), (w,1,1) \in \PA_{3}(\y)$. 
We claim that:
\begin{enumerate}[label=(\arabic*)]
    \item\label{it: mi triple (1)}
    $(1,1,w) \in \PA_{3}(\y)$ if and only if $w\leq 1+y_1+y_2$.
    \item\label{it: mi triple (2)} 
    $(1,w,1) \in \PA_{3}(\y)$ if and only if
    \begin{enumerate}[label=(2\alph*)]
        \item\label{it: mi triple (2a)} $w\leq 1+y_1$, or
        \item\label{it: mi triple (2b)} $w=1+y_1+y_3$.
    \end{enumerate}
    \item\label{it: mi triple (3)}
    $(w,1,1) \in \PA_{3}(\y)$ if and only if
    \begin{enumerate}[label=(3\alph*)]
        \item\label{it: mi triple (3a)} $w=1+y_2$, or
        \item\label{it: mi triple (3b)} $w=1+y_2+y_3$, or
        \item\label{it: mi triple (3c)} $y_2\geq y_3$ and $w=1+y_3$.
    \end{enumerate}
\end{enumerate}
To see \ref{it: mi triple (1)}, note that $(1,1,w)$ is nondecreasing and so the claim follows from Proposition~\ref{proposition: nondecreasing}.
To see \ref{it: mi triple (2)}, note that if $(1,w,1) \in \PA_3(\y)$, then car $2$ either parks immediately after car $1$, or parks after car $1$ while leaving a gap (of unoccupied spots) of length exactly the length of car $3$.
The former holds whenever $w \leq 1 + y_1$ (this is \ref{it: mi triple (2a)}), whereas the latter holds if $w = 1 + y_1 + y_3$ (this is \ref{it: mi triple (2b)}).
To see \ref{it: mi triple (3)}, note that if $(w,1,1) \in \PA_3(\y)$, then car $1$ either leaves a gap (of unoccupied spots) of length exactly the length of car $2$, leaves a gap of length exactly the sum of lengths of car $2$ and car $3$, or leaves a gap of length exactly the length of car $3$ which car $2$ is unable to occupy.
The first case holds if $w = 1 + y_2$ (this is \ref{it: mi triple (3a)}), the second case holds if $w = 1 + y_2 + y_3$ (this is \ref{it: mi triple (3b)}), and the third case holds if $y_2 \geq y_3$ and $w = 1 + y_3$ (this is \ref{it: mi triple (3c)}).

We remark that the conditions above generate a solution set for $(y_1,y_2,y_3)$, which is the complement of the conditions in the theorem statement.
Therefore, $\y$ is minimally invariant if and only if there does not exist $w \in \N_{> 1}$ satisfying the conditions in \ref{it: mi triple (1)}, \ref{it: mi triple (2)}, and \ref{it: mi triple (3)}.
There are six different ways in which the conditions can hold:
\begin{itemize}
    \item \ref{it: mi triple (1)}, \ref{it: mi triple (2a)}, and \ref{it: mi triple (3a)} hold, or
    \item \ref{it: mi triple (1)}, \ref{it: mi triple (2a)}, and \ref{it: mi triple (3b)} hold, or
    \item \ref{it: mi triple (1)}, \ref{it: mi triple (2a)}, and \ref{it: mi triple (3c)} hold, or
    \item \ref{it: mi triple (1)}, \ref{it: mi triple (2b)}, and \ref{it: mi triple (3a)} hold, or
    \item \ref{it: mi triple (1)}, \ref{it: mi triple (2b)}, and \ref{it: mi triple (3b)} hold, or
    \item \ref{it: mi triple (1)}, \ref{it: mi triple (2b)}, and \ref{it: mi triple (3c)} hold.
\end{itemize}

Suppose $y_1 < y_2$, $y_1 < y_3$, and $y_1 + y_3 \neq y_2$.
We claim that, in this case, there does not exist $w \in \N_{> 1}$ such that the conditions in \ref{it: mi triple (1)}, \ref{it: mi triple (2)}, and \ref{it: mi triple (3)} hold.
To prove this, we show that if \ref{it: mi triple (3)} holds, then it cannot be that both \ref{it: mi triple (1)} and \ref{it: mi triple (2)} hold.
If \ref{it: mi triple (3a)} holds, then \ref{it: mi triple (2a)} cannot hold for otherwise $1 + y_2 = w \leq 1 + y_1$, implying $y_2 \leq y_1$ and contradicting $y_1 < y_2$.
Similarly, if \ref{it: mi triple (3a)} holds, then \ref{it: mi triple (2b)} cannot hold for otherwise $1 + y_2 = w = 1 + y_1 + y_3$, implying $y_2 = y_1 + y_3$ and contradicting $y_1 + y_3 \neq y_2$.
This shows that if \ref{it: mi triple (3a)} holds, then \ref{it: mi triple (1)} and \ref{it: mi triple (2)} cannot both hold.
If \ref{it: mi triple (3b)} holds, then \ref{it: mi triple (1)} cannot hold for otherwise $1 + y_2 + y_3 = w \leq 1 + y_1 + y_2$, implying $y_3 \leq y_1$ and contradicting $y_1 < y_3$.
This shows that if \ref{it: mi triple (3b)} holds, then \ref{it: mi triple (1)} and \ref{it: mi triple (2)} cannot both hold.
If \ref{it: mi triple (3c)} holds, then \ref{it: mi triple (2a)} cannot hold for otherwise $1 + y_3 = w \leq 1 + y_1$, implying $y_3 \leq y_1$ and contradicting $y_1 < y_3$.
Similarly, if \ref{it: mi triple (3c)} holds, then \ref{it: mi triple (2b)} cannot hold for otherwise $1 + y_3 = w = 1 + y_1 + y_3$, implying $y_1 = 0$ and contradicting $(y_1, y_2, y_3) \in \N^3$.
This shows that if \ref{it: mi triple (3c)} holds, then \ref{it: mi triple (1)} and \ref{it: mi triple (2)} cannot both hold.
This shows that if \ref{it: mi triple (3)} holds, then \ref{it: mi triple (1)} and \ref{it: mi triple (2)} cannot both hold.

Conversely, suppose it is not the case that $y_1 < y_2$, $y_1 < y_3$, and $y_1 + y_3 \neq y_2$.
Then, at least one of $y_1 \geq y_2$, $y_1 \geq y_3$, or $y_1 + y_3 = y_2$ holds.
If $y_1 \geq y_2$, then letting $w = 1 + y_2$ satisfies $w > 1$, $w \leq 1 + y_1 + y_2$ and so \ref{it: mi triple (1)} holds, $w \leq 1 + y_1$ and so \ref{it: mi triple (2a)} holds, and $w = 1 + y_2$ and so \ref{it: mi triple (3a)} holds.
This shows that, in this case, there exists $w \in \N_{> 1}$ such that the inequality conditions in \ref{it: mi triple (1)}, \ref{it: mi triple (2)}, and \ref{it: mi triple (3)} hold.
If $y_1 < y_2$ and $y_1 \geq y_3$, then $y_2 \geq y_3$ and letting $w = 1 + y_3$ satisfies $w > 1$, $w \leq 1 + y_2 \leq 1 + y_1 + y_2$ and so \ref{it: mi triple (1)} holds, $w \leq 1 + y_1$ and so \ref{it: mi triple (2a)} holds, and $y_2 \geq y_3$ and $w = 1 + y_3$ and so \ref{it: mi triple (3c)} holds.
This shows that, in this case, there exists $w \in \N_{> 1}$ such that the inequality conditions in \ref{it: mi triple (1)}, \ref{it: mi triple (2)}, and \ref{it: mi triple (3)} hold.
Lastly, if $y_1 < y_2$, $y_1 < y_3$, and $y_2 = y_1 + y_3$ , then letting $w = 1 + y_2$ satisfies $w > 1$, $w \leq 1 + y_1 + y_2$ and so \ref{it: mi triple (1)} holds, $w = 1 + y_1 + y_3$ and so \ref{it: mi triple (2b)} holds, and $w = 1 + y_2$ and so \ref{it: mi triple (3a)} holds.
This shows that, in this case, there exists $w \in \N_{> 1}$ such that the inequality conditions in \ref{it: mi triple (1)}, \ref{it: mi triple (2)}, and \ref{it: mi triple (3)} hold.
This shows that if it is not the case that $y_1 < y_2$, $y_1 < y_3$, and $y_1 + y_3 \neq y_2$, then there exists $w \in \N_{> 1}$ such that the inequality conditions in \ref{it: mi triple (1)}, \ref{it: mi triple (2)}, and \ref{it: mi triple (3)} hold.
\end{proof}

Next, we provide a full characterization of the set of invariant parking assortments with three cars.
Given the technicality of these results, we provide the proofs in the listed appendix. 

\begin{theorem}
\label{theorem: invariant triple}
Let $a<b<c$ be in $\N$. Then Table \ref{thm:table} provides car lengths $\y \in \{a,b,c\}^3$ and the corresponding sets $\PAINVND_3(\y)$.
\begin{table}[ht]
\resizebox{\textwidth}{!}{
\begin{tabular}{|c|c|c|}\hline
$\y$      & $\PAINVND_3(\y)$             & Proof                 \\ \hline
        $(a,a,a)$ & $(1,1,1),(1,1,1+a),(1,1,1+2a),(1,1+a,1+a),(1,1+a,1+2a)$ & \ref{prop: y=(a,a,a)} \\
        \hline
        $(a,a,b)$ & $(1,1,1),(1,1,1+a)$ & \ref{prop: y=(a,a,b)} \\
        \hline
        $(a,b,a)$, $b=2a$ & $(1,1,1),(1,1,1+a),(1,1,1+2a)$ & \ref{prop: y=(a,b,a), b=2a} \\
        \hline
        $(a,b,a)$, $b\neq 2a$ & $(1,1,1),(1,1,1+a)$ & \ref{prop: y=(a,b,a), b!=2a} \\ 
        \hline 
        $(b,a,a)$, $2a\leq b$ & $(1,1,1),(1,1,1+a),(1,1,1+2a),(1,1+a,1+a),(1,1+a,1+2a)$ & \ref{prop: y=(b,a,a), 2a<=b} \\ 
        \hline 
        $(b,a,a)$, $2a>b$ & $(1,1,1),(1,1,1+a),(1,1+a,1+a)$ & \ref{prop: y=(b,a,a), 2a>b} \\
        \hline
        $(a,b,b)$ & $(1,1,1)$ & \ref{prop: y=(a,b,b)} \\
        \hline 
        $(b,a,b)$ & $(1,1,1),(1,1,1+a)$ & \ref{prop: y=(b,a,b)} \\ 
        \hline 
        $(b,b,a)$ & $(1,1,1),(1,1,1+a),(1,1,1+b),(1,1,1+a+b)$ & \ref{prop: y=(b,b,a)} \\ 
        \hline
        $(a,b,c)$ & $(1,1,1)$ & \ref{prop: y=(a,b,c)} \\
        \hline
        $(a,c,b)$, $a+b=c$ & $(1,1,1),(1,1,1+a+b)$ & \ref{prop: y=(a,c,b), a+b=c} \\ 
        \hline 
        $(a,c,b)$, $a+b\neq c$ & $(1,1,1)$ & \ref{prop: y=(a,c,b), a+b!=c} \\ 
        \hline 
        $(b,a,c)$ & $(1,1,1),(1,1,1+a)$ & \ref{prop: y=(b,a,c)} \\
        \hline 
        $(b,c,a)$, $a+b=c$ & $(1,1,1),(1,1,1+a),(1,1,1+a+b)$ & \ref{prop: y=(b,c,a), a+b=c} \\
        \hline 
        $(b,c,a)$, $a+b\neq c$ & $(1,1,1),(1,1,1+a)$ & \ref{prop: y=(b,c,a), a+b!=c} \\
        \hline 
        $(c,a,b)$, $a+b\leq c$ & $(1,1,1),(1,1,1+a),(1,1,1+a+b)$ & \ref{prop: y=(c,a,b), a+b<=c} \\
        \hline 
        $(c,a,b)$, $a+b>c$ & $(1,1,1),(1,1,1+a)$ & \ref{prop: y=(c,a,b), a+b>c} \\
        \hline 
        $(c,b,a)$, $a+b\leq c$ & $(1,1,1),(1,1,1+a),(1,1,1+b),(1,1,1+a+b)$ & \ref{prop: y=(c,b,a), a+b<=c} \\
        \hline 
        $(c,b,a)$, $a+b>c$ & $(1,1,1),(1,1,1+a),(1,1,1+b)$ & \ref{prop: y=(c,b,a), a+b>c} \\
        \hline 
\end{tabular}
}
\caption{Car lengths $\y \in \{a,b,c\}^3$ with $a<b<c$ and corresponding sets $\PAINVND_3(\y)$.
}\label{thm:table}
\end{table}
\end{theorem}

\begin{remark}
    Note that for any arbitrary $\y\in\N^3$, there exists a choice of $a,b,c$ so that $\y$ is represented in one of the rows of Table \ref{thm:table}. Hence, the list of cases presented in Theorem \ref{theorem: invariant triple} is exhaustive.
\end{remark}

The proof of Theorem \ref{theorem: invariant triple}, presented in Appendix~\ref{sec: appendix proof of theorem invariant trible}, leverages the following general result, which we present here in order to highlight its use.
\begin{theorem}
\label{theorem: unique 1}
Let $\y=(y_1,y_2,\ldots,y_n) \in \mathbb{N}^n$ and $\x=(x_1,x_2,\ldots,x_n) \in \N^n$. 
If $\x\in\PAINV_n(\y)$ and $|\{j \in [n] : x_j = 1\}| = 1$, then $\beta_j(\x) \in \PAINV_{n-1}(\y_{\widehat{i}})$ for every $i,j\in[n]$, where $\beta_j: \N^n \rightarrow \N^{n-1}$ is given by $\beta_j(\x) = (b_1, \ldots, b_{j-1}, b_{j+1}, \ldots, b_n)$, where for $k \in [n] \setminus \{j\}$ we let $b_k = \max\{1,x_k - y_j\}$.
\end{theorem}
\begin{proof}
Throughout this proof we assume $\x \in \PAINV_n(\y)$ and $|\{j \in [n] : x_j = 1\}| = 1$, as in the theorem statement.
For $i \in [n]$ define $\alpha_i: \N^n \rightarrow \N^n$ by $\alpha_i(\x)= (a_1,a_2, \ldots, a_n)$, where for car $k \in [n]$ we let
\begin{align*}
    a_k =
    \begin{cases}
    x_i, \mbox{ if $k=i$} \\
    \max\{1+y_i,x_k\}, \mbox{ otherwise}.
    \end{cases}
\end{align*}
We begin by proving the following claims.
\begin{claim}
\label{claim: unique 1 (1)}
For any $i \in [n]$, if $\x' = (x_1', x_2', \ldots, x_n')$ is a rearrangement of $\x$ with $x_i' = 1$, then $\alpha_i(\x')\in\PA_n(\y)$.
\end{claim}
\begin{cproof}
Let $\x'$ be as in the claim statement.
Note that $\x\in\PAINV_n(\y)$ implies $\x'\in\PA_n(\y)$.
Moreover, since $|\{j \in [n] : x_j = 1\}| = 1$, $\x'\in\PA_n(\y)$ implies car $i$ under $\x'$ parks in spots $1, 2, \ldots, y_i$.
Therefore, reassigning the preference of any car $k \in [n] \setminus \{i\}$ with $x_k' \in [y_i]$ into spot $1 + y_i$ does not affect the order in which the cars park.
That is, $\alpha_i(\x')\in\PA_n(\y)$.
\end{cproof}

\begin{claim}
\label{claim: unique 1 (2)}
For any $i \in [n]$ and any rearrangement $\x'$ of $\x$ with $x_i' = 1$, if $\alpha_i(\x')\in\PA_n(\y)$, then $\beta_i(\x')\in\PA_{n-1}(\y_{\widehat{i}})$.
\end{claim}
\begin{cproof}
Note that by construction of $\alpha_i$ there does not exist a car $k \in [n] \setminus \{i\}$ with preference ${\alpha_i(\x)}_k \in [y_i]$. 
Therefore, removing the $i$th entry from ${\alpha_i(\x)}$ and subtracting $y_i$ from each entry ${\alpha_i(\x')}_k$ with $k \in [n] \setminus \{i\}$ yields $\beta_i(\x')$.
Recall that if $\alpha_i(\x') \in \PA_n(\y)$, car $i$ under $\alpha_i(\x')$ parks in spots $1,2,\ldots,y_i$, and hence cars $k \in [n] \setminus \{i\}$ under $\alpha_i(\x')$ fully occupy spots $1 + y_i,2 + y_i,\ldots,\sum_{k=1}^n y_k$. 
If so, then under $\beta_i(\x')$, those same cars fully occupy spots $1, 2, \ldots, (\sum_{j=1}^n y_j) - y_i$.
Note that this is just a shift and reindexing of the parking spots and car preferences by the same value. 
This shows $\beta_i(\x') \in \PA_{n-1}(\y_{\widehat{i}})$.
\end{cproof}

Claim~\ref{claim: unique 1 (1)} and \ref{claim: unique 1 (2)} together imply that if $\x \in \PAINV_n(\y)$, then for any $i \in [n]$ and any rearrangement $\x'$ of $\x$ with $x_i' = 1$, $\beta_i(\x')\in\PA_{n-1}(\y_{\widehat{i}})$. 
Moreover, since this holds for any $i \in [n]$ and rearrangement $\x'$ of $\x$ with $x_i' = 1$, we have that if $\x\in\PAINV_n(\y)$, then $\beta_i(\x) \in \PAINV_{n-1}(\y_{\widehat{i}})$ for every $i \in [n]$. 
To complete the proof we now show the following.
\begin{claim}
\label{claim: unique 1 (3)}
For any $
i, j \in [n]$, if $\beta_i(\x) \in \PAINV_{n-1}(\y_{\widehat{i}})$, then $\beta_j(\x) \in \PAINV_{n-1}(\y_{\widehat{i}})$.
\end{claim}
\begin{cproof}
Note that if $i=j$, the statement holds. 
Fix any distinct $i, j \in [n]$ and let $\pi(\x)$ be the rearrangement of $\x$ that swaps its $i$th and $j$th entries and keeps the remaining entries fixed, i.e.,~$\pi$ is the transposition $(i,j)$ and $\pi(\x)=
(x_1, \ldots, x_{i-1}, x_j, x_{i+1}, \ldots, x_{j-1}, x_i, x_{j+1}, \ldots, x_n)$, where for notation we assume $i<j$ (note that this is immaterial to the argument).
Since $\x \in \PAINV_n(\y)$, we know $\pi(\x) \in \PAINV_n(\y)$, which implies $\beta_i(\pi(\x)) \in \PAINV_{n-1}(\y_{\widehat{i}})$ for all $i\in[n]$. 
Now note that by the definitions of $\pi(\x)$, $\beta_i$, and $\beta_j$ we have that $\beta_i(\pi(\x)) = \beta_j(\x)$ for any $i,j\in[n]$.
\end{cproof}

This completes the proof.
\end{proof}
To illustrate Theorem~\ref{theorem: unique 1} we present the following.
\begin{example}
Let $\y=(2,2,2)$ and $\x=(1,3,5).$ 
Note that $|\{j \in [n] : x_j = 1\}| = 1$, and $\x\in\PAINV_3(\y)$ by Theorem~\ref{theorem: constant}. 
So $\x$ satisfies the hypotheses of Theorem \ref{theorem: unique 1}. We have $\y_{\widehat{1}}=\y_{\widehat{2}}=\y_{\widehat{3}}=(2,2)$
    and
$\beta_1(\x)=\beta_2(\x)=(1,3)$, and $\beta_3(\x)=(1,1)$.
    Since $(1,3),(1,1)\in \PAINV_3((2,2))$ by Theorem~\ref{theorem: constant}, we have $\beta_j(\x)\in\PAINV_2(\y_{\widehat{i}})$ for all $i,j\in[3].$
\end{example}

\section{Open Problems}
\label{section: open problems}
In this section we provide some directions for future study.

\subsection{Connections Between Parking Assortments and Sequences}

Given that we have extended the definition of parking sequences to parking assortments, it would be interesting to further explore the connections between these sets of objects. 
To begin one could provide a characterization of non-trivial car lengths $\y\in\N^n$ such that $\PS_n(\y)=\PA_n(\y)$. 
Moreover, in terms of minimally invariant car lengths, we ask:
When does $\PSINV_n(\y)=\{(1^n)\}$ imply $\PAINV_n(\y)=\{(1^n)\}$?

\subsection{Parking Outcomes}
In~\cite{colmenarejo2021counting}, the authors provide product formulas for the number of $k$-Naples parking functions (when $k=0$, these are classical parking functions) by enumerating those which result in cars parked in a certain order. We then ask:

\begin{openproblem}\label{open:perms}
Fix $\y\in\N^n$ and let $\sigma=\sigma_1\sigma_2\cdots\sigma_n\in\mathfrak{S}_n$  denote the order in which the cars ultimately park. That is, $\sigma_i=j$ means that car $j$ was the $i$th car parked on the street. What is the number of parking assortments (or parking sequences) $\x$ which park the cars in order $\sigma$?
\end{openproblem}
Note that an answer to Open Problem \ref{open:perms} would yield a sum formula (over all permutations) giving a full count for the number of parking assortments (or parking sequences). 

Since the preliminary version of this work, Open Problem~\ref{open:perms} has been solved by Franks, Harris, Harry, Kretschmann, and Vance~\cite{franks2023counting}.

\subsection{Boolean Formula Characterizations}

While Theorem~\ref{theorem: mi} provides a pseudopolynomial-time characterization of minimally invariant car lengths, it relies on oracle calls to the ``parking experiment.''
An arguably more expressive characterization could be obtained through a Boolean formula; Corollary~\ref{corollary: mi pair} and Theorem~\ref{theorem: mi triple} establish these for minimally invariant car lengths with two and three cars, respectively.
A natural follow-up to these results would be the corresponding characterization with four cars.
Our computational experiments suggest the following.
\begin{conjecture}
\label{conjecture: mi quadruple}
Let $\y=(y_1,y_2,y_3,y_4)\in\N^4$.
Then, $\y$ is minimally invariant if and only if the following hold
\begin{align*}
    & (y_1 < y_2) 
    \land (y_1 < y_3)
    \land (y_1 < y_4) 
    \land (y_2 \neq y_1 + y_3)  
    \land (y_2 \neq y_1 + y_3 + y_4) \\
    & \quad \land ((y_2 < y_1 + y_3) \lor (y_3 \neq y_1 + y_4)) \\
    & \quad \land ((y_2 > y_1 + y_3) \lor ((y_2 \neq y_1 + y_4) \land ((y_2 < y_3) \lor (y_3 \neq y_1 + y_4)))).
\end{align*}
\end{conjecture}

Note that the clauses in Corollary~\ref{corollary: mi pair} and Theorem~\ref{theorem: mi triple} form a subset of the clauses in Conjecture~\ref{conjecture: mi quadruple}.
Hence, more generally, we would like to understand the following.
\begin{openproblem}
Give a characterization of the    recursive nature of Boolean formulas for $\PAINV_n(\y)$ and the growth rate of their size for $n \in \N$. 
\end{openproblem}

\subsection{Computational Complexity}

The question we study in this work is: given $\y \in \N^n$ and $\x \in [\sum_{i = 1}^n y_i]$, is $\x \in \PAINV_n(\y)$?
Here we refer to this decision problem as \texttt{INV-PARKING-ASMT($\x,\y)$}.
Note that, from a computational point of view, $\x \in \PA_n(\y)$ can be easily decided by conducting the corresponding ``parking experiment.''
Therefore, \texttt{INV-PARKING-ASMT($\x,\y)$} can be decided in $O(n!)$ time by conducting the parking experiment for every rearrangement $\x'$ of $\x$.
Can this worst-case bound be improved to time polynomial in $n$?

For the case in which $\y = (c^n) \in \N^n$, Theorem~\ref{theorem: constant} answers this question positively.
Similarly, when $\y$ satisfies the conditions for minimal invariance in Theorem~\ref{theorem: mi}, we can conclude $\x \in \PAINV_n(\y)$  if and only if $\x=(1^n).$
In addition, Theorem~\ref{theorem: invariant pair} and Theorem~\ref{theorem: invariant triple} provide a full characterization for the settings in which there are two or three cars, respectively.
Beyond these special cases, however, a concise characterization deciding $\x \in \PAINV_n(\y)$ in its full generality remains elusive.
Is it possible that there is no concise characterization?
We formalize this possibility as follows.

A decision problem $\Pi$ is in co-NP if ``\texttt{No}'' answers have a deterministic polynomial-time verifier.
A decision problem $\Pi$ is co-NP-complete if all co-NP problems can be reduced to it in polynomial time.
Note that \texttt{INV-PARKING-ASMT($\x,\y)$} is in co-NP.
To see this, note that any rearrangement $\x'$ of $\x$ for which $\x' \notin \PA_n(\y)$, as determined by the ``parking experiment,'' certifies $\x \notin \PAINV_n(\y)$.
We now ask the following.
\begin{openproblem}
Is \texttt{INV-PARKING-ASMT($\x,\y)$} co-NP-complete?
\end{openproblem}
No polynomial-time algorithm is known for a co-NP-complete problem.
In particular, $\text{P} \stackrel{?}{=} \text{co-NP}$ is a well-known open problem.
Therefore, if \texttt{INV-PARKING-ASMT($\x,\y)$} is co-NP-complete, there is no concise characterization for \texttt{INV-PARKING-ASMT($\x,\y)$} unless $\text{P} = \text{co-NP}$.
A potential starting point in showing that \texttt{INV-PARKING-ASMT($\x,\y)$} is co-NP-complete might be the reduction in~\cite{uznanski2015all} for the co-NP-completeness of the all-permutations supersequence problem.

\subsection{Permutation Subsets}
Another problem to consider is whether preference sequences are invariant under subsets of the symmetric group. 
To formalize this, let $\y=(y_{1},y_2,\ldots,y_{n})\in\N^n$ and $\x = (x_{1},x_2,\ldots,x_{n})\in\PA_n(\y)$.
Let $\pi\in\mathfrak{S}_n$ be a permutation on the indices of $\x$. 
That is, define $\pi(\x)=(x_{\pi(1)},x_{\pi(2)},\ldots,x_{\pi(n)})$.
We say that $\x$ is \emph{$\pi$-invariant} if $\pi(\x)\in \PA_n(\y)$. More generally, for any subset (or subgroup) $T$ of the symmetric groups $\mathfrak{S_n}$, we say that $\x$ is $T$-invariant if $\pi(\x)\in\PA_n(\y)$ for all $\pi\in T$.
Of course, when $T=\mathfrak{S}_n$, $T$-invariant is precisely what we have studied. Moreover, if $\y\in\N^n$ and $\x\in\PAINV_n(\y)$, then $\x$ is $T$-invariant for any $T\subseteq\mathfrak{S}_n$. In what follows, for every $i\in[n-1]$, we let $s_{i}$ denote the neighboring transposition swapping indices $i$ and $i+1$. Hence, if $\x=(x_1,x_2,\ldots,x_n)$, then $s_i(\x)=(x_1,x_2,\ldots,x_{i-1},x_{i+1},x_{i},x_{i+2},\ldots,x_n)$. 
\begin{proposition}\label{prop:converse fails}
    Let $\y=(y_1,y_2,\ldots,y_n)\in\N^n$, $\x=(x_1,x_2,\ldots,x_n)\in\PA_n(\y)$, and $\mathcal{I}=\{i_1,i_2,\ldots,i_k\} \subseteq [n-1]$. 
    If $\x$ is $T$-invariant where $T=\{\pi\in \mathfrak{S}_n: \pi=\prod_{j\in\mathcal{J}} s_{j}\mbox{ and $\mathcal{J}\subseteq\mathcal{I}$}\}$, then $\x$ is $s_i$-invariant for all $i\in\mathcal{I}$.
\end{proposition}
\begin{proof}
As noted before, this result follows directly from the definition of $\x$ being $T$-invariant and taking $\mathcal{J}$ as the singleton sets $\{s_i\}$ for each $i\in\mathcal{I}$.
\end{proof}

We note that the converse of Proposition~\ref{prop:converse fails} is not generally true. 
For example, consider $\y=(1,2,2)$ and $\x=(1,1,2)\in\PA_3(\y)$. Then 
 $\x = s_1(\x)=(1,1,2)\in\PA_3(\y)$, and 
 $s_2(\x)=s_2s_1(\x)=(1,2,1)\in\PA_3(\y)$.
However, $s_1s_2(\x)=s_1s_2s_1(\x)=(2,1,1)\not\in\PA_3(\y)$.

In the special case where $\pi\in\mathfrak{S}_n$ is a product of disjoint transpositions, namely $\pi=s_{i_1}s_{i_2}\cdots s_{i_k}$ with $i_1,i_2,\ldots,i_k$ being nonconsecutive integers in the set $[n-1]$, one might believe that if $\x\in\PA_n(\y)$ is $s_{i_j}$-invariant for all $1\leq j\leq k$, then $\x$ is $\pi$-invariant. 
This is also false.
As an example, consider $\y=(1,2,1,2)$ and  $\x=(1,2,1,2)\in\PA_4(\y)$. 
Then $s_1(\x)=(2,1,1,2)\in\PA_4(\y)$ and $s_3(\x)=(1,2,2,1)\in\PA_4(\y)$.
However, $s_1s_3(\x)=(2,1,2,1)\not\in\PA_4(\y)$. 
In light of this, we ask the following.

\begin{openproblem}\label{last problem}
Suppose that $\x\in\PA_n(\y)$ is both $s_{i}$-invariant and $s_{j}$-invariant for some distinct $i,j\in[n-1]$. Then what must be true of $\y$ and $\x$ so that $\x$ is $s_{i} s_{j}$-invariant,  or
    $\x$ is $s_{j} s_{i}$-invariant?
Moreover, under what conditions on $\y$, $\x$, and $\mathcal{I}\subset[n-1]$ does $\x$ being $s_i$-invariant for all $i\in\mathcal{I}$ guarantee that $\x$ is $\pi$-invariant, where $\pi=\prod_{i\in \mathcal{I}}s_i$?
Lastly, what must be true about $\y$ and $\x$ so that if $\x$ is $s_i$-invariant for all $i\in[n-1]$, then $\x\in\PAINV_n(\y)$?
\end{openproblem}

We suspect that a good entryway into Open Problem~\ref{last problem} is to consider small values of $n$ and also the case where $\y=(c^n)$.

\section*{Acknowledgements}
This material is based upon work supported by the National Science Foundation under Grant No.~DMS-1929284 while the authors were in residence at the Institute for Computational and Experimental Research in Mathematics in Providence, RI.

\printbibliography

\addresseshere

\newpage
\appendix

\section{Proof of Theorem~\ref{theorem: invariant triple}}
\label{sec: appendix proof of theorem invariant trible}

In what follows, let $a,b,c \in \N$ with $a<b<c$ and recall \[\PAINVND_n(\y) := \{\x = (x_1, x_2, \ldots, x_n) \in \PAINV_n(\y) : x_1 \leq x_2 \leq \cdots \leq x_n \}\] denotes the set of nondecreasing invariant parking assortments given $\y \in \N^n$.
Moreover, since $(1^n) \in \PAINV_n(\y)$ for any $\y \in \N^n$, in our proofs of the following results we only argue about $\x \in \PAINV_n(\y)$ with $\x \neq (1^n)$.
\begin{proposition}
\label{prop: y=(a,a,a)}
Let $\y=(a,a,a) \in \N^3$. 
Then, $\PAINVND_3(\y)= \{(1,1,1), (1,1,1+a), (1,1,1+2a), (1,1+a,1+a), (1,1+a,1+2a)\}$.
\end{proposition}
\begin{proof}
Let $\x=(x_1,x_2,x_3)\in[3a]^3$ be nondecreasing. 
By Theorem~\ref{theorem: constant}, $\x=(x_1,x_2,x_3)\in\PAINV_3(\y)$ if and only if $x_1,x_2,x_3\equiv 1 \mod{a}$, there is at least one index $j\in[3]$ such that $x_j\leq a$, there are at least two indices $j\in[3]$ such that $x_j\leq 2a$, and $x_j\leq 3a$ for all $j\in[3]$. 
These statements imply that $x_{1} = 1$, $x_{2} \in \{1, 1+a\}$, and $x_{3} \in \{1, 1+a, 1+2a\}$, as claimed.
\end{proof}

\begin{proposition}
\label{prop: y=(a,a,b)}
Let $\y = (a, a, b) \in \N^3$.
Then, $\PAINVND_3(\y) = \{(1,1,1), (1,1,1+a)\}$.
\end{proposition}
\begin{proof}
Let $\x = (x_1, x_2, x_3) \in \PAINVND_3(\y)$.
Clearly $x_1 = 1$.
We now show that $x_2=1$ by contradiction assuming that $x_2>1$ and considering the following two mutually exclusive possibilities:
\begin{itemize}
    \item[] 
    Case 1: Suppose $x_2 > a$.
    Then, Theorem~\ref{theorem: unique 1} implies $(x_2 - a, x_3 - a) \in \PAINV_2((a,b))$.
    Since $a < b$, Corollary~\ref{corollary: mi pair} implies $(a,b)$ is minimally invariant, and so $x_2 = x_3 = 1 + a$.
    However, $\x = (1, 1 + a, 1 + a) \notin \PAINV_3(\y)$ since its rearrangement $\x' = (1 + a, 1 + a, 1) \notin \PA_3(\y)$.
    \item[] 
    Case 2: Suppose $x_2 \leq a$.
    If $a = 1$ we have $x_2 = 1$, contradicting the assumption that $x_2 > 1$.
    If $a > 1$, Lemma~\ref{lemma: minentry} implies $\x \notin \PAINV_3(\y)$.
\end{itemize}
Therefore, if $\x \in \PAINV_3(\y)$ is nondecreasing, then $x_1 = x_2 = 1$.
By the proof of Theorem~\ref{theorem: mi triple}, upon substituting $\x = (1, 1, w)$ and $\y = (a, a, b)$, we have that $\x \in \PAINV_3(\y)$ for $w \in \N_{> 1}$ if and only if
\begin{enumerate}[label=(\arabic*)]
    \item $w\leq 1+2a$ holds, and
    \item at least one of
    \begin{enumerate}[label=(2\alph*)]
        \item $w\leq 1+a$, or
        \item $w=1+a+b$ holds, and
    \end{enumerate}
    \item at least one of
    \begin{enumerate}[label=(3\alph*)]
        \item $w=1+a$, or
        \item $w=1+a+b$ holds.
    \end{enumerate}
\end{enumerate}
(We omit \ref{it: mi triple (3c)} since it requires $a \geq b$, a contradiction.)
To find the valid solutions for $w$, consider the following table, where ``None'' indicates that selecting the conditions marked ``Yes'' leads to no solutions for $w$. 
Namely, in the table, each row corresponds to a combination of ways in which conditions (1), (2), and (3) may hold, and the solutions for $w$ that arise (if any).
Note that empty cells represent that the listed condition is immaterial to that row.
\begin{align*}
    \begin{array}{|c|c|c|c|c|c|}
        \hline
        \text{Cond. (1)} & \text{Cond. (2a)} & \text{Cond. (2b)} & \text{Cond. (3a)} & \text{Cond. (3b)} & w \\
        \hline
        \text{Yes} & \text{Yes} & & \text{Yes} & & 1+a \\
        \hline
        \text{Yes} & \text{Yes} & & & \text{Yes} & \text{None} \\
        \hline
        \text{Yes} & & \text{Yes} & \text{Yes} & & \text{None} \\
        \hline
        \text{Yes} & & \text{Yes} & & \text{Yes} & \text{None} \\
        \hline 
    \end{array}
\end{align*}
Therefore, $w = 1 + a$ is the only valid solution satisfying $w > 1$.
This shows that $\PAINVND_3(\y) = \{(1,1,1), (1,1,1+a)\}$.
\end{proof}
    
\begin{proposition}
\label{prop: y=(a,b,a), b=2a}
Let $\y=(a,b,a) \in \N^3$ with $b=2a$. 
Then, $\PAINVND_3(\y)= \{ (1,1,1), (1,1,1+a), (1,1,1+2a)\}$.
\end{proposition}
\begin{proof}
By Theorem~\ref{theorem: mi triple}, $\y$ is not minimally invariant. 
Therefore, by Theorem \ref{theorem: mi}, there exists $w \in \N_{> 1}$ such that $(1,1,w) \in \PAINV_3(\y)$.
Note moreover that $\y_{\vert_2} = (a, b)$ is minimally invariant by Theorem~\ref{corollary: mi pair}.
Therefore, by Corollary~\ref{corollary: append}, all $\x \in \PAINV_n(\y)$ are in fact of the form $\x = (1, 1, w)$ for some $w \in \N_{> 1}$.

By the proof of Theorem~\ref{theorem: mi triple}, upon substituting $\x = (1, 1, w)$ and $\y = (a, b, a) = (a, 2a, a)$, we have that $\x \in \PAINV_3(\y)$ for $w \in \N_{> 1}$ if and only if
\begin{enumerate}[label=(\arabic*)]
    \item $w\leq 1+3a$ holds, and
    \item at least one of
    \begin{enumerate}[label=(2\alph*)]
        \item $w\leq 1+a$, or
        \item $w=1+2a$ holds, and
    \end{enumerate}
    \item at least one of
    \begin{enumerate}[label=(3\alph*)]
        \item $w=1+2a$, or
        \item $w=1+3a$, or
        \item $2a\geq a$ and $w=1+a$ holds.
    \end{enumerate}
\end{enumerate}
To find the valid solutions for $w$, consider the following table, where ``None'' indicates that selecting the conditions marked ``Yes'' leads to no solutions for $w$. 
Namely, in the table, each row corresponds to a combination of ways in which conditions (1), (2), and (3) may hold, and the solutions for $w$ that arise (if any).
Note that empty cells represent that the listed condition is immaterial to that row.
\begin{align*}
    \begin{array}{|c|c|c|c|c|c|c|}
        \hline
        \text{Cond. (1)} & \text{Cond. (2a)} & \text{Cond. (2b)} & \text{Cond. (3a)} & \text{Cond. (3b)} &  \text{Cond. (3c)}& w \\
        \hline
        \text{Yes} & \text{Yes} & & \text{Yes} & & & \text{None} \\
        \hline
        \text{Yes} & \text{Yes} & & & \text{Yes} & & \text{None} \\
        \hline
        \text{Yes} & \text{Yes} & & & & \text{Yes} & 1+a \\
        \hline
        \text{Yes} & & \text{Yes} & \text{Yes} & & & 1+2a \\
        \hline
        \text{Yes} & & \text{Yes} & & \text{Yes} & & \text{None} \\
        \hline
        \text{Yes} & & \text{Yes} & & & \text{Yes} & \text{None} \\
        \hline 
    \end{array}
\end{align*}
Therefore, $w = 1 + a$ and $w = 1 + 2a$ are the only valid solutions satisfying $w > 1$.
This shows that $\PAINVND_3(\y)= \{ (1,1,1), (1,1,1+a), (1,1,1+2a)\}$.
\end{proof}

\begin{proposition}
\label{prop: y=(a,b,a), b!=2a}
Let $\y=(a,b,a) \in \N^3$ with $b\neq 2a$. 
Then, $\PAINVND_3(\y)= \{ (1,1,1), (1,1,1+a)\}$.
\end{proposition}
\begin{proof}
By Theorem~\ref{theorem: mi triple}, $\y$ is not minimally invariant. 
Therefore, by Theorem \ref{theorem: mi}, there exists $w \in \N_{> 1}$ such that $(1,1,w) \in \PAINV_3(\y)$.
Note moreover that $\y_{\vert_2} = (a, b)$ is minimally invariant by Theorem~\ref{corollary: mi pair}.
Therefore, by Corollary~\ref{corollary: append}, all $\x \in \PAINV_n(\y)$ are in fact of the form $\x = (1, 1, w)$ for some $w \in \N_{> 1}$.

By the proof of Theorem~\ref{theorem: mi triple}, upon substituting $\x = (1, 1, w)$ and $\y = (a, b, a) \neq (a, 2a, a)$, we have that $\x \in \PAINV_3(\y)$ for $w \in \N_{> 1}$ if and only if
\begin{enumerate}[label=(\arabic*)]
    \item $w\leq 1+a+b$ holds, and
    \item at least one of
    \begin{enumerate}[label=(2\alph*)]
        \item $w\leq 1+a$, or
        \item $w=1+2a$ holds, and
    \end{enumerate}
    \item at least one of
    \begin{enumerate}[label=(3\alph*)]
        \item $w=1+b$, or
        \item $w=1+a+b$, or
        \item $b\geq a$ and $w=1+a$ holds.
    \end{enumerate}
\end{enumerate}
To find the valid solutions for $w$, consider the following table, where ``None'' indicates that selecting the conditions marked ``Yes'' leads to no solutions for $w$. 
Namely, in the table, each row corresponds to a combination of ways in which conditions (1), (2), and (3) may hold, and the solutions for $w$ that arise (if any).
Note that empty cells represent that the listed condition is immaterial to that row.
\begin{align*}
    \begin{array}{|c|c|c|c|c|c|c|}
        \hline
        \text{Cond. (1)} & \text{Cond. (2a)} & \text{Cond. (2b)} & \text{Cond. (3a)} & \text{Cond. (3b)} &  \text{Cond. (3c)}& w \\
        \hline
        \text{Yes} & \text{Yes} & & \text{Yes} & & & \text{None} \\
        \hline
        \text{Yes} & \text{Yes} & & & \text{Yes} & & \text{None} \\
        \hline
        \text{Yes} & \text{Yes} & & & & \text{Yes} & 1+a \\
        \hline
        \text{Yes} & & \text{Yes} & \text{Yes} & & & \text{None} \\
        \hline
        \text{Yes} & & \text{Yes} & & \text{Yes} & & \text{None} \\
        \hline
        \text{Yes} & & \text{Yes} & & & \text{Yes} & \text{None} \\
        \hline 
    \end{array}
\end{align*}
Therefore, $w = 1 + a$ is the only valid solution satisfying $w > 1$.
This shows that $\PAINVND_3(\y)= \{ (1,1,1), (1,1,1+a)\}$.
\end{proof}

\begin{proposition}
\label{prop: y=(b,a,a), 2a<=b}
Let $\y=(b,a,a) \in \N^3$ with $2a\leq b$. 
Then, $\PAINVND_3(\y)= \{(1,1,1), (1,1,1+a), (1,1,1+2a), (1,1+a,1+a), (1,1+a,1+2a)\}$.
\end{proposition}
\begin{proof}
Let $\x = (x_1, x_2, x_3) \in \PAINV_3(\y)$ be nondecreasing.
Clearly $x_1 = 1$.
We first assume $x_2, x_3 > 1$ and find the assignments of $x_2, x_3$ satisfying $\x \in \PAINV_3(\y)$.
There are four mutually exclusive possibilities:
\begin{itemize}
    \item[] 
    Case 1: Suppose $b < x_2 \leq x_3$.
    Then, Theorem~\ref{theorem: unique 1} implies $(x_2 - b, x_3 -b) \in \PAINV_2((a,a))$ and $(x_2 - a, x_3 - a) \in \PAINV_2((b,a))$.
    The former implies $x_2 = 1 + b$, whereas the latter implies $x_2 = 1 + a$, a contradiction.
    \item[] 
    Case 2: Suppose $a < x_2 \leq b < x_3$.
    Then, Theorem~\ref{theorem: unique 1} implies $(1, x_3 -b) \in \PAINV_2((a,a))$ and $(x_2 - a, x_3 - a) \in \PAINV_2((b,a))$.
    By Theorem~\ref{theorem: invariant pair}, the former implies either $x_3 = 1 + b$ or $x_3 = 1 + a + b$, whereas the latter implies $x_2 = 1 + a$ and either $x_3 = 1 + a$ or $x_3 = 1 + 2a$.
    All such assignments of $x_2, x_3$ except $x_2 = 1 + a$ and $x_3 = 1 + 2a$ with $2a = b$ yield a contradiction. 
    One can easily verify that $(1,1+a,1+2a)\in\PAINV_3(\y)$ when $2a=b.$
    \item[] 
    Case 3: Suppose $a < x_2 \leq x_3 \leq b$.
    Then, Theorem~\ref{theorem: unique 1} implies $(1, 1) \in \PAINV_2((a,a))$ and $(x_2 - a, x_3 - a) \in \PAINV_2((b,a))$.
    By Theorem~\ref{theorem: invariant pair}, the latter implies $x_2 = 1 + a$ and either $x_3 = 1 + a$ or $x_3 = 1 + 2a$.
    One can easily verify that $(1,1+a,1+a)\in\PAINV_3(\y)$ and $(1,1+a,1+2a)\in\PAINV_3(\y)$.
    \item[] 
    Case 4: Suppose $x_2 \leq a$.
    If $a = 1$ we have $x_2 = 1$, contradicting the assumption that $x_2 > 1$.
    If $a > 1$, Lemma~\ref{lemma: minentry} implies $\x \notin \PAINV_3(\y)$.
\end{itemize}
Therefore, if $\x \in \PAINV_3(\y)$ and $x_2, x_3 > 1$, then $x_1 = 1, x_2 = 1 + a$, and either $x_3 = 1 + a$ or $x_3 = 1 + 2a$.

Next, we assume $x_2 = 1$ and $x_3 > 1$, and find the assignment of $x_3$ satisfying $\x \in \PAINV_3(\y)$.
By the proof of Theorem~\ref{theorem: mi triple}, upon substituting $\x = (1, 1, w)$ and $\y = (b, a, a)$ with $2a \leq b$, we have that $\x \in \PAINV_3(\y)$ for $w \in \N_{> 1}$ if and only if
\begin{enumerate}[label=(\arabic*)]
    \item $w\leq 1+a+b$ holds, and
    \item at least one of
    \begin{enumerate}[label=(2\alph*)]
        \item $w\leq 1+b$, or
        \item $w=1+a+b$ holds, and
    \end{enumerate}
    \item at least one of
    \begin{enumerate}[label=(3\alph*)]
        \item $w=1+a$, or
        \item $w=1+2a$, or
        \item $a\geq a$ and $w=1+a$ holds.
    \end{enumerate}
\end{enumerate}
To find the valid solutions for $w$, consider the following table, where ``None'' indicates that selecting the conditions marked ``Yes'' leads to no solutions for $w$. 
Namely, in the table, each row corresponds to a combination of ways in which conditions (1), (2), and (3) may hold, and the solutions for $w$ that arise (if any).
Note that empty cells represent that the listed condition is immaterial to that row.
\begin{align*}
    \begin{array}{|c|c|c|c|c|c|c|}
        \hline
        \text{Cond. (1)} & \text{Cond. (2a)} & \text{Cond. (2b)} & \text{Cond. (3a)} & \text{Cond. (3b)} &  \text{Cond. (3c)}& w \\
        \hline
        \text{Yes} & \text{Yes} & & \text{Yes} & & & 1+a \\
        \hline
        \text{Yes} & \text{Yes} & & & \text{Yes} & & 1+2a \\
        \hline
        \text{Yes} & \text{Yes} & & & & \text{Yes} & 1+a \\
        \hline
        \text{Yes} & & \text{Yes} & \text{Yes} & & & \text{None} \\
        \hline
        \text{Yes} & & \text{Yes} & & \text{Yes} & & \text{None} \\
        \hline
        \text{Yes} & & \text{Yes} & & & \text{Yes} &  \text{None}\\
        \hline 
    \end{array}
\end{align*}
Therefore, $w = 1 + a$ and $w = 1 + 2a$ are the only valid solutions satisfying $w > 1$.
This shows that $\PAINVND_3(\y)= \{(1,1,1), (1,1,1+a), (1,1,1+2a), (1,1+a,1+a), (1,1+a,1+2a)\}$.
\end{proof}

\begin{proposition}
\label{prop: y=(b,a,a), 2a>b}
Let $\y=(b,a,a) \in \N^3$ with $2a>b$.
Then, $\PAINVND_3(\y)=\{(1,1,1), (1,1,1+a), (1,1+a,1+a)\}$.
\end{proposition}
\begin{proof}
Let $\x = (x_1, x_2, x_3) \in \PAINV_3(\y)$ be nondecreasing.
Clearly $x_1 = 1$.
We first assume $x_2, x_3 > 1$ and find the assignments of $x_2, x_3$ satisfying $\x \in \PAINV_3(\y)$.
There are four mutually exclusive possibilities:
\begin{itemize}
    \item[] 
    Case 1: Suppose $b < x_2 \leq x_3$.
    Then, Theorem~\ref{theorem: unique 1} implies $(x_2 - b, x_3 -b) \in \PAINV_2((a,a))$ and $(x_2 - a, x_3 - a) \in \PAINV_2((b,a))$.
    The former implies $x_2 = 1 + b$, whereas the latter implies $x_2 = 1 + a$, a contradiction.
    \item[] 
    Case 2: Suppose $a < x_2 \leq b < x_3$.
    Then, Theorem~\ref{theorem: unique 1} implies $(1, x_3 -b) \in \PAINV_2((a,a))$ and $(x_2 - a, x_3 - a) \in \PAINV_2((b,a))$.
    By Theorem~\ref{theorem: invariant pair}, the former implies either $x_3 = 1 + b$ or $x_3 = 1 + a + b$, whereas the latter implies $x_2 = 1 + a$ and either $x_3 = 1 + a$ or $x_3 = 1 + 2a$.
    All such assignments of $x_2, x_3$ yield a contradiction since $2a > b$.
    \item[] 
    Case 3: Suppose $a < x_2 \leq x_3 \leq b$.
    Then, Theorem~\ref{theorem: unique 1} implies $(1, 1) \in \PAINV_2((a,a))$ and $(x_2 - a, x_3 - a) \in \PAINV_2((b,a))$.
    By Theorem~\ref{theorem: invariant pair}, the latter implies $x_2 = 1 + a$ and either $x_3 = 1 + a$ or $x_3 = 1 + 2a$.
    Assigning $x_3 = 1 + 2a$ yields a contradiction since $2a > b$, so the only valid assignment is $x_2 = x_3 = 1 + a$.
    One can easily verify that $(1,1+a,1+a)\in\PAINV_3(\y)$.
    \item[] 
    Case 4: Suppose $x_2 \leq a$.
    If $a = 1$ we have $x_2 = 1$, contradicting the assumption that $x_2 > 1$.
    If $a > 1$, Lemma~\ref{lemma: minentry} implies $\x \notin \PAINV_3(\y)$.
\end{itemize}
Therefore, if $\x \in \PAINV_3(\y)$ and $x_2, x_3 > 1$, then $x_1 = 1, x_2 = 1 + a$, and $x_3 = 1 + a$.

Next, we assume $x_2 = 1$ and $x_3 > 1$, and find the assignment of $x_3$ satisfying $\x \in \PAINV_3(\y)$.
By the proof of Theorem~\ref{theorem: mi triple}, upon substituting $\x = (1, 1, w)$ and $\y = (b, a, a)$ with $2a > b$, we have that $\x \in \PAINV_3(\y)$ for $w \in \N_{> 1}$ if and only if
\begin{enumerate}[label=(\arabic*)]
    \item $w\leq 1+a+b$ holds, and
    \item at least one of
    \begin{enumerate}[label=(2\alph*)]
        \item $w\leq 1+b$, or
        \item $w=1+a+b$ holds, and
    \end{enumerate}
    \item at least one of
    \begin{enumerate}[label=(3\alph*)]
        \item $w=1+a$, or
        \item $w=1+2a$, or
        \item $a\geq a$ and $w=1+a$ holds.
    \end{enumerate}
\end{enumerate}
To find the valid solutions for $w$, consider the following table, where ``None'' indicates that selecting the conditions marked ``Yes'' leads to no solutions for $w$. 
Namely, in the table, each row corresponds to a combination of ways in which conditions (1), (2), and (3) may hold, and the solutions for $w$ that arise (if any).
Note that empty cells represent that the listed condition is immaterial to that row.
\begin{align*}
    \begin{array}{|c|c|c|c|c|c|c|}
          \hline
        \text{Cond. (1)} & \text{Cond. (2a)} & \text{Cond. (2b)} & \text{Cond. (3a)} & \text{Cond. (3b)} &\text{Cond. (3c)}& w \\
       \hline
      \text{Yes} & \text{Yes} & & \text{Yes} & & & 1+a \\
        \hline
        \text{Yes} & \text{Yes} & & & \text{Yes} & & \text{None} \\
        \hline
        \text{Yes} & \text{Yes} & & & & \text{Yes} & 1+a \\
        \hline
        \text{Yes} & & \text{Yes} & \text{Yes} & & & \text{None} \\
        \hline
        \text{Yes} & & \text{Yes} & & \text{Yes} & & \text{None} \\
        \hline
        \text{Yes} & & \text{Yes} & & & \text{Yes} &  \text{None}\\
        \hline 
    \end{array}
\end{align*}
Therefore, $w = 1 + a$ is the only valid solution satisfying $w > 1$.
This shows that $\PAINVND_3(\y)= \{(1,1,1), (1,1,1+a), (1,1+a,1+a)\}$.
\end{proof}

\begin{proposition}
\label{prop: y=(a,b,b)}
Let $\y=(a,b,b) \in \N^3$. 
Then, $\PAINVND_3(\y)= \{(1,1,1)\}$.
\end{proposition}
\begin{proof}
The result follows from Theorem~\ref{theorem: mi triple}.
\end{proof}

\begin{proposition}
\label{prop: y=(b,a,b)}
Let $\y=(b,a,b) \in \N^3$. 
Then, $\PAINVND_3(\y)= \{(1,1,1), (1,1,1+a)\}$.
\end{proposition}
\begin{proof}
Let $\x = (x_1, x_2, x_3) \in \PAINVND_3(\y)$.
Clearly $x_1 = 1$.
We now show that $x_2=1$ by contradiction assuming that $x_2>1$ and considering the following two mutually exclusive possibilities:
\begin{itemize}
    \item[] 
    Case 1: Suppose $x_2 > a$.
    Then, Theorem~\ref{theorem: unique 1} implies $(x_2 - a, x_3 - a) \in \PAINV_2((b,b))$.
    By Theorem~\ref{theorem: invariant pair}, this implies $x_2 = 1 + a$ and either $x_3 = 1 + a$ or $x_3 = 1 + a + b$.
    However, $\x = (1, 1 + a, 1 + a) \notin \PAINV_3(\y)$ since its rearrangement $\x' = (1 + a, 1 + a, 1) \notin \PA_3(\y)$.
    Similarly, $\x = (1, 1 + a, 1 + a + b) \notin \PAINV_3(\y)$  since its rearrangement $\x' = (1 + a, 1 + a + b, 1) \notin \PA_3(\y)$.
    \item[] 
    Case 2: Suppose $x_2 \leq a$.
    If $a = 1$ we have $x_2 = 1$, contradicting the assumption that $x_2 > 1$.
    If $a > 1$, Lemma~\ref{lemma: minentry} implies $\x \notin \PAINV_3(\y)$.
\end{itemize}
Therefore, if $\x \in \PAINV_3(\y)$ is nondecreasing, then $x_1 = x_2 = 1$.
By the proof of Theorem~\ref{theorem: mi triple}, upon substituting $\x = (1, 1, w)$ and $\y = (b, a, b)$, we have that $\x \in \PAINV_3(\y)$ for $w \in \N_{> 1}$ if and only if
\begin{enumerate}[label=(\arabic*)]
    \item $w\leq 1+a+b$ holds, and
    \item at least one of
    \begin{enumerate}[label=(2\alph*)]
        \item $w\leq 1+b$, or
        \item $w=1+2b$ holds, and
    \end{enumerate}
    \item at least one of
    \begin{enumerate}[label=(3\alph*)]
        \item $w=1+a$, or
        \item $w=1+a+b$ holds.
    \end{enumerate}
\end{enumerate}
(We omit \ref{it: mi triple (3c)} since it requires $a \geq b$, a contradiction.)
To find the valid solutions for $w$, consider the following table, where ``None'' indicates that selecting the conditions marked ``Yes'' leads to no solutions for $w$. 
Namely, in the table, each row corresponds to a combination of ways in which conditions (1), (2), and (3) may hold, and the solutions for $w$ that arise (if any).
Note that empty cells represent that the listed condition is immaterial to that row.
\begin{align*}
    \begin{array}{|c|c|c|c|c|c|}
        \hline
       \text{Cond. (1)} & \text{Cond. (2a)} & \text{Cond. (2b)} & \text{Cond. (3a)} & \text{Cond. (3b)} & w \\
         \hline
        \text{Yes} & \text{Yes} & & \text{Yes} & & 1+a \\
        \hline
        \text{Yes} & \text{Yes} & & & \text{Yes} & \text{None} \\
        \hline
        \text{Yes} & & \text{Yes} & \text{Yes} & & \text{None} \\
        \hline
        \text{Yes} & & \text{Yes} & & \text{Yes} & \text{None} \\
        \hline 
    \end{array}
\end{align*}
Therefore, $w = 1 + a$ is the only valid solution satisfying $w > 1$.
This shows that $\PAINVND_3(\y) = \{(1,1,1), (1,1,1+a)\}$.
\end{proof}

\begin{proposition}
\label{prop: y=(b,b,a)}
Let $\y=(b,b,a) \in \N^3$. 
Then, $\PAINVND_3(\y)=\{(1,1,1), (1,1,1+a), (1,1,1+b), (1,1,1+a+b)\}$.
\end{proposition}
\begin{proof}
Let $\x = (x_1, x_2, x_3) \in \PAINV_3(\y)$ be nondecreasing.
Clearly $x_1 = 1$.
We now show that $x_2=1$ by contradiction assuming that $x_2>1$ and considering the following two mutually exclusive possibilities:
\begin{itemize}
    \item[] 
    Case 1: Suppose $x_2 > a$.
    Then, Theorem~\ref{theorem: unique 1} implies $(x_2 - a, x_3 - a) \in \PAINV_2((b,b))$.
    By Theorem~\ref{theorem: invariant pair}, this implies $x_2 = 1 + a$ and either $x_3 = 1 + a$ or $x_3 = 1 + a + b$.
    However, $\x = (1, 1 + a, 1 + a) \notin \PAINV_3(\y)$ since its rearrangement $\x' = (1 + a, 1, 1 + a) \notin \PA_3(\y)$.
    Similarly, $\x = (1, 1 + a, 1 + a + b) \notin \PAINV_3(\y)$  since its rearrangement $\x' = (1 + a, 1, 1 + a + b) \notin \PA_3(\y)$.
    \item[] 
    Case 2: Suppose $x_2 \leq a$.
    If $a = 1$ we have $x_2 = 1$, contradicting the assumption that $x_2 > 1$.
    If $a > 1$, Lemma~\ref{lemma: minentry} implies $\x \notin \PAINV_3(\y)$.
\end{itemize}
Therefore, if $\x \in \PAINV_3(\y)$ is nondecreasing, then $x_1 = x_2 = 1$.
By the proof of Theorem~\ref{theorem: mi triple}, upon substituting $\x = (1, 1, w)$ and $\y = (b, b, a)$, we have that $\x \in \PAINV_3(\y)$ for $w \in \N_{> 1}$ if and only if
\begin{enumerate}[label=(\arabic*)]
    \item $w\leq 1+2b$ holds, and
    \item at least one of
    \begin{enumerate}[label=(2\alph*)]
        \item $w\leq 1+b$, or
        \item $w=1+a+b$ holds, and
    \end{enumerate}
    \item at least one of
    \begin{enumerate}[label=(3\alph*)]
        \item $w=1+b$, or
        \item $w=1+a+b$, or
        \item $b\geq a$ and $w=1+a$ holds.
    \end{enumerate}
\end{enumerate}
To find the valid solutions for $w$, consider the following table, where ``None'' indicates that selecting the conditions marked ``Yes'' leads to no solutions for $w$. 
Namely, in the table, each row corresponds to a combination of ways in which conditions (1), (2), and (3) may hold, and the solutions for $w$ that arise (if any).
Note that empty cells represent that the listed condition is immaterial to that row.
\begin{align*}
    \begin{array}{|c|c|c|c|c|c|c|}
        \hline
        \text{Cond. (1)} & \text{Cond. (2a)} & \text{Cond. (2b)} & \text{Cond. (3a)} & \text{Cond. (3b)} &\text{Cond. (3c)}& w \\
       \hline
        \text{Yes} & \text{Yes} & & \text{Yes} & & & 1+b \\
        \hline
        \text{Yes} & \text{Yes} & & & \text{Yes} & & \text{None} \\
        \hline
        \text{Yes} & \text{Yes} & & & & \text{Yes} & 1+a \\
        \hline
        \text{Yes} & & \text{Yes} & \text{Yes} & & & \text{None} \\
        \hline
        \text{Yes} & & \text{Yes} & & \text{Yes} & & 1+a+b \\
        \hline
        \text{Yes} & & \text{Yes} & & & \text{Yes} &  \text{None}\\
        \hline 
    \end{array}
\end{align*}
Therefore, $w = 1 + a$, $w = 1 + b$, and $w = 1 + a + b$ are the only valid solutions satisfying $w > 1$.
This shows that $\PAINVND_3(\y)=\{(1,1,1), (1,1,1+a), (1,1,1+b), (1,1,1+a+b)\}$.
\end{proof}

\begin{proposition}
\label{prop: y=(a,b,c)}
Let $\y=(a,b,c) \in \N^3$. 
Then, $\PAINVND_3(\y)=\{(1,1,1)\}$.
\end{proposition}
\begin{proof}
The result follows from Theorem \ref{theorem: mi triple}.
\end{proof}

\begin{proposition}
\label{prop: y=(a,c,b), a+b=c}
Let $\y=(a,c,b) \in \N^3$ with $a+b=c$. 
Then, $\PAINVND_3(\y)=\{(1,1,1), (1,1,1+a+b)\}$.
\end{proposition}
\begin{proof}
By Theorem~\ref{theorem: mi triple}, $\y$ is not minimally invariant. 
Therefore, by Theorem \ref{theorem: mi}, there exists $w \in \N_{> 1}$ such that $(1,1,w) \in \PAINV_3(\y)$.
Note moreover that $\y_{\vert_2} = (a, c)$ is minimally invariant by Theorem~\ref{corollary: mi pair}.
Therefore, by Corollary~\ref{corollary: append}, all $\x \in \PAINV_n(\y)$ are in fact of the form $\x = (1, 1, w)$ for some $w \in \N_{> 1}$.

By the proof of Theorem~\ref{theorem: mi triple}, upon substituting $\x = (1, 1, w)$ and $\y = (a, c, b) = (a, a + b, b)$, we have that $\x \in \PAINV_3(\y)$ for $w \in \N_{> 1}$ if and only if
\begin{enumerate}[label=(\arabic*)]
    \item $w\leq 1+2a+b$ holds, and
    \item at least one of
    \begin{enumerate}[label=(2\alph*)]
        \item $w\leq 1+a$, or
        \item $w=1+a+b$ holds, and
    \end{enumerate}
    \item at least one of
    \begin{enumerate}[label=(3\alph*)]
        \item $w=1+a+b$, or
        \item $w=1+a+2b$, or
        \item $a+b\geq b$ and $w=1+b$ holds.
    \end{enumerate}
\end{enumerate}
To find the valid solutions for $w$, consider the following table, where ``None'' indicates that selecting the conditions marked ``Yes'' leads to no solutions for $w$. 
Namely, in the table, each row corresponds to a combination of ways in which conditions (1), (2), and (3) may hold, and the solutions for $w$ that arise (if any).
Note that empty cells represent that the listed condition is immaterial to that row.
\begin{align*}
    \begin{array}{|c|c|c|c|c|c|c|}
        \hline
        \text{Cond. (1)} & \text{Cond. (2a)} & \text{Cond. (2b)} & \text{Cond. (3a)} & \text{Cond. (3b)} &\text{Cond. (3c)}& w \\
       \hline
        \text{Yes} & \text{Yes} & & \text{Yes} & & & \text{None} \\
        \hline
        \text{Yes} & \text{Yes} & & & \text{Yes} & & \text{None} \\
        \hline
        \text{Yes} & \text{Yes} & & & & \text{Yes} & \text{None} \\
        \hline
        \text{Yes} & & \text{Yes} & \text{Yes} & & & 1+a+b \\
        \hline
        \text{Yes} & & \text{Yes} & & \text{Yes} & & \text{None} \\
        \hline
        \text{Yes} & & \text{Yes} & & & \text{Yes} & \text{None} \\
        \hline 
    \end{array}
\end{align*}
Therefore, $w = 1 + a + b$ is the only valid solution satisfying $w > 1$.
This shows that $\PAINVND_3(\y) = \{(1,1,1), (1,1,1+a+b)\}$.
\end{proof}

\begin{proposition}
\label{prop: y=(a,c,b), a+b!=c}
Let $\y=(a,c,b) \in \N^3$ with $a+b\neq c$. 
Then, $\PAINVND_3(\y)=\{(1,1,1)\}$.
\end{proposition}
\begin{proof}
The result follows from Theorem \ref{theorem: mi triple}.
\end{proof}

\begin{proposition}
\label{prop: y=(b,a,c)}
Let $\y=(b,a,c) \in \N^3$. 
Then, $\PAINVND_3(\y)=\{(1,1,1), (1,1,1+a)\}$.
\end{proposition}
\begin{proof}
Let $\x = (x_1, x_2, x_3) \in \PAINV_3(\y)$ be nondecreasing.
Clearly $x_1 = 1$.
We now show that $x_2=1$ by contradiction assuming that $x_2>1$ and considering the following two mutually exclusive possibilities:
\begin{itemize}
    \item[] 
    Case 1: Suppose $x_2 > a$.
    Then, Theorem~\ref{theorem: unique 1} implies $(x_2 - a, x_3 - a) \in \PAINV_2((b,c))$.
    Since $b < c$, Corollary~\ref{corollary: mi pair} implies $(b,c)$ is minimally invariant, and so $x_2 = x_3 = 1 + a$.
    However, $\x = (1, 1 + a, 1 + a) \notin \PAINV_3(\y)$ since its rearrangement $\x' = (1 + a, 1 + a, 1) \notin \PA_3(\y)$.
    \item[] 
    Case 2: Suppose $x_2 \leq a$.
    If $a = 1$ we have $x_2 = 1$, contradicting the assumption that $x_2 > 1$.
    If $a > 1$, Lemma~\ref{lemma: minentry} implies $\x \notin \PAINV_3(\y)$.
\end{itemize}
Therefore, if $\x \in \PAINV_3(\y)$ is nondecreasing, then $x_1 = x_2 = 1$.
By the proof of Theorem~\ref{theorem: mi triple}, upon substituting $\x = (1, 1, w)$ and $\y = (b, a, c)$, we have that $\x \in \PAINV_3(\y)$ for $w \in \N_{> 1}$ if and only if
\begin{enumerate}[label=(\arabic*)]
    \item $w\leq 1+a+b$ holds, and
    \item at least one of
    \begin{enumerate}[label=(2\alph*)]
        \item $w\leq 1+b$, or
        \item $w=1+b+c$ holds, and
    \end{enumerate}
    \item at least one of
    \begin{enumerate}[label=(3\alph*)]
        \item $w=1+a$, or
        \item $w=1+a+c$ holds.
    \end{enumerate}
\end{enumerate}
(We omit \ref{it: mi triple (3c)} since it requires $a \geq c$, a contradiction.)
To find the valid solutions for $w$, consider the following table, where ``None'' indicates that selecting the conditions marked ``Yes'' leads to no solutions for $w$. 
Namely, in the table, each row corresponds to a combination of ways in which conditions (1), (2), and (3) may hold, and the solutions for $w$ that arise (if any).
Note that empty cells represent that the listed condition is immaterial to that row.
\begin{align*}
    \begin{array}{|c|c|c|c|c|c|}
        \hline
       \text{Cond. (1)} & \text{Cond. (2a)} & \text{Cond. (2b)} & \text{Cond. (3a)} & \text{Cond. (3b)} & w \\
        \hline
        \text{Yes} & \text{Yes} & & \text{Yes} & & 1+a \\
        \hline
        \text{Yes} & \text{Yes} & & & \text{Yes} & \text{None} \\
        \hline
        \text{Yes} & & \text{Yes} & \text{Yes} & & \text{None} \\
        \hline
        \text{Yes} & & \text{Yes} & & \text{Yes} & \text{None} \\
        \hline 
    \end{array}
\end{align*}
Therefore, $w = 1 + a$ is the only valid solution satisfying $w > 1$.
This shows that $\PAINVND_3(\y) = \{(1,1,1), (1,1,1+a)\}$.
\end{proof}

\begin{proposition}
\label{prop: y=(b,c,a), a+b=c}
Let $\y=(b,c,a) \in \N^3$ with $a+b=c$. 
Then, $\PAINVND_3(\y)=\{(1,1,1),(1,1,1+a),(1,1,1+a+b)\}$.
\end{proposition}
\begin{proof}
By Theorem~\ref{theorem: mi triple}, $\y$ is not minimally invariant. 
Therefore, by Theorem \ref{theorem: mi}, there exists $w \in \N_{> 1}$ such that $(1,1,w) \in \PAINV_3(\y)$.
Note moreover that $\y_{\vert_2} = (b, c)$ is minimally invariant by Theorem~\ref{corollary: mi pair}.
Therefore, by Corollary~\ref{corollary: append}, all $\x \in \PAINV_n(\y)$ are in fact of the form $\x = (1, 1, w)$ for some $w \in \N_{> 1}$.

By the proof of Theorem~\ref{theorem: mi triple}, upon substituting $\x = (1, 1, w)$ and $\y = (b, c, a) = (b, a + b, a)$, we have that $\x \in \PAINV_3(\y)$ for $w \in \N_{> 1}$ if and only if
\begin{enumerate}[label=(\arabic*)]
    \item $w\leq 1+a+2b$ holds, and
    \item at least one of
    \begin{enumerate}[label=(2\alph*)]
        \item $w\leq 1+b$, or
        \item $w=1+a+b$ holds, and
    \end{enumerate}
    \item at least one of
    \begin{enumerate}[label=(3\alph*)]
        \item $w=1+a+b$, or
        \item $w=1+2a+b$, or
        \item $a+b\geq a$ and $w=1+a$ holds.
    \end{enumerate}
\end{enumerate}
To find the valid solutions for $w$, consider the following table, where ``None'' indicates that selecting the conditions marked ``Yes'' leads to no solutions for $w$. 
Namely, in the table, each row corresponds to a combination of ways in which conditions (1), (2), and (3) may hold, and the solutions for $w$ that arise (if any).
Note that empty cells represent that the listed condition is immaterial to that row.
\begin{align*}
    \begin{array}{|c|c|c|c|c|c|c|}
        \hline
        \text{Cond. (1)} & \text{Cond. (2a)} & \text{Cond. (2b)} & \text{Cond. (3a)} & \text{Cond. (3b)}&\text{Cond. (3c)} & w \\
        \hline
        \text{Yes} & \text{Yes} & & \text{Yes} & & & \text{None} \\
        \hline
        \text{Yes} & \text{Yes} & & & \text{Yes} & & \text{None} \\
        \hline
        \text{Yes} & \text{Yes} & & & & \text{Yes} & 1+a \\
        \hline
        \text{Yes} & & \text{Yes} & \text{Yes} & & & 1+a+b \\
        \hline
        \text{Yes} & & \text{Yes} & & \text{Yes} & & \text{None} \\
        \hline
        \text{Yes} & & \text{Yes} & & & \text{Yes} & \text{None} \\
        \hline 
    \end{array}
\end{align*}
Therefore, $w = 1 + a$ and $w = 1 + a + b$ are the only valid solutions satisfying $w > 1$.
This shows that $\PAINVND_3(\y)=\{(1,1,1), (1,1,1+a),(1,1,1+a+b)\}$.
\end{proof}

\begin{proposition}
\label{prop: y=(b,c,a), a+b!=c}
Let $\y=(b,c,a) \in \N^3$ with $a+b\neq c$. 
Then, $\PAINVND_3(\y)=\{(1,1,1),(1,1,1+a)\}$.
\end{proposition}
\begin{proof}
By Theorem~\ref{theorem: mi triple}, $\y$ is not minimally invariant. 
Therefore, by Theorem \ref{theorem: mi}, there exists $w \in \N_{> 1}$ such that $(1,1,w) \in \PAINV_3(\y)$.
Note moreover that $\y_{\vert_2} = (b, c)$ is minimally invariant by Theorem~\ref{corollary: mi pair}.
Therefore, by Corollary~\ref{corollary: append}, all $\x \in \PAINV_n(\y)$ are in fact of the form $\x = (1, 1, w)$ for some $w \in \N_{> 1}$.

By the proof of Theorem~\ref{theorem: mi triple}, upon substituting $\x = (1, 1, w)$ and $\y = (b, c, a) \neq (b, a + b, a)$, we have that $\x \in \PAINV_3(\y)$ for $w \in \N_{> 1}$ if and only if
\begin{enumerate}[label=(\arabic*)]
    \item $w\leq 1+b+c$ holds, and
    \item at least one of
    \begin{enumerate}[label=(2\alph*)]
        \item $w\leq 1+b$, or
        \item $w=1+a+b$ holds, and
    \end{enumerate}
    \item at least one of
    \begin{enumerate}[label=(3\alph*)]
        \item $w=1+c$, or
        \item $w=1+a+c$, or
        \item $c\geq a$ and $w=1+a$ holds.
    \end{enumerate}
\end{enumerate}
To find the valid solutions for $w$, consider the following table, where ``None'' indicates that selecting the conditions marked ``Yes'' leads to no solutions for $w$. 
Namely, in the table, each row corresponds to a combination of ways in which conditions (1), (2), and (3) may hold, and the solutions for $w$ that arise (if any).
Note that empty cells represent that the listed condition is immaterial to that row.
\begin{align*}
    \begin{array}{|c|c|c|c|c|c|c|}
        \hline 
        \text{Cond. (1)} & \text{Cond. (2a)} & \text{Cond. (2b)} & \text{Cond. (3a)} & \text{Cond. (3b)} &\text{Cond. (3c)} & w \\
       \hline
        \text{Yes} & \text{Yes} & & \text{Yes} & & & \text{None} \\
        \hline
        \text{Yes} & \text{Yes} & & & \text{Yes} & & \text{None} \\
        \hline
        \text{Yes} & \text{Yes} & & & & \text{Yes} & 1+a \\
        \hline
        \text{Yes} & & \text{Yes} & \text{Yes} & & & \text{None} \\
        \hline
        \text{Yes} & & \text{Yes} & & \text{Yes} & & \text{None} \\
        \hline
        \text{Yes} & & \text{Yes} & & & \text{Yes} & \text{None} \\
        \hline 
    \end{array}
\end{align*}
Therefore, $w = 1 + a$ is the only valid solution satisfying $w > 1$.
This shows that $\PAINVND_3(\y)=\{(1,1,1), (1,1,1+a)\}$.
\end{proof}

\begin{proposition}
\label{prop: y=(c,a,b), a+b<=c}
Let $\y=(c,a,b) \in \N^3$ with $a+b\leq c$. 
Then, $\PAINVND_3(\y)=\{(1,1,1),(1,1,1+a),(1,1,1+a+b)\}$.
\end{proposition}
\begin{proof}
Let $\x = (x_1, x_2, x_3) \in \PAINV_3(\y)$ be nondecreasing.
Clearly $x_1 = 1$.
We now show that $x_2=1$ by contradiction assuming that $x_2>1$ and considering the following two mutually exclusive possibilities:
\begin{itemize}
    \item[] 
    Case 1: Suppose $x_2 > a$.
    Then, Theorem~\ref{theorem: unique 1} implies $(x_2 - a, x_3 - a) \in \PAINV_2((c,b))$.
    By Theorem~\ref{theorem: invariant pair}, this implies $x_2 = 1 + a$ and either $x_3 = 1 + a$ or $x_3 = 1 + a + b$.
    However, $\x = (1, 1 + a, 1 + a) \notin \PAINV_3(\y)$ since its rearrangement $\x' = (1 + a, 1 + a, 1) \notin \PA_3(\y)$.
    Similarly, $\x = (1, 1 + a, 1 + a + b) \notin \PAINV_3(\y)$  since its rearrangement $\x' = (1 + a, 1 + a + b, 1) \notin \PA_3(\y)$.
    \item[] 
    Case 2: Suppose $x_2 \leq a$.
    If $a = 1$ we have $x_2 = 1$, contradicting the assumption that $x_2 > 1$.
    If $a > 1$, Lemma~\ref{lemma: minentry} implies $\x \notin \PAINV_3(\y)$.
\end{itemize}
Therefore, if $\x \in \PAINV_3(\y)$ is nondecreasing, then $x_1 = x_2 = 1$.
By the proof of Theorem~\ref{theorem: mi triple}, upon substituting $\x = (1, 1, w)$ and $\y = (c, a, b)$ with $a + b\leq c$, we have that $\x \in \PAINV_3(\y)$ for $w \in \N_{> 1}$ if and only if
\begin{enumerate}[label=(\arabic*)]
    \item $w\leq 1+a+c$ holds, and
    \item at least one of
    \begin{enumerate}[label=(2\alph*)]
        \item $w\leq 1+c$, or
        \item $w=1+b+c$ holds, and
    \end{enumerate}
    \item at least one of
    \begin{enumerate}[label=(3\alph*)]
        \item $w=1+a$, or
        \item $w=1+a+b$ holds.
    \end{enumerate}
\end{enumerate}
(We omit \ref{it: mi triple (3c)} since it requires $a \geq b$, a contradiction.)
To find the valid solutions for $w$, consider the following table, where ``None'' indicates that selecting the conditions marked ``Yes'' leads to no solutions for $w$. 
Namely, in the table, each row corresponds to a combination of ways in which conditions (1), (2), and (3) may hold, and the solutions for $w$ that arise (if any).
Note that empty cells represent that the listed condition is immaterial to that row.
\begin{align*}
    \begin{array}{|c|c|c|c|c|c|}
        \hline
        \text{Cond. (1)} & \text{Cond. (2a)} & \text{Cond. (2b)} & \text{Cond. (3a)} & \text{Cond. (3b)}  & w \\\hline
        \text{Yes} & \text{Yes} & & \text{Yes} & & 1+a \\
        \hline
        \text{Yes} & \text{Yes} & & & \text{Yes} & 1+a+b \\
        \hline
        \text{Yes} & & \text{Yes} & \text{Yes} & & \text{None} \\
        \hline
        \text{Yes} & & \text{Yes} & & \text{Yes} & \text{None} \\
        \hline 
    \end{array}
\end{align*}
Therefore, $w = 1 + a$ and $w = 1 + a + b$ are the only valid solutions satisfying $w > 1$.
This shows that $\PAINVND_3(\y)=\{(1,1,1), (1,1,1+a), (1,1,1+a+b)\}$.
\end{proof}

\begin{proposition}
\label{prop: y=(c,a,b), a+b>c}
Let $\y=(c,a,b) \in \N^3$ with $a+b>c$. 
Then, $\PAINVND_3(\y)=\{(1,1,1),(1,1,1+a)\}$.
\end{proposition}
\begin{proof}
Let $\x = (x_1, x_2, x_3) \in \PAINV_3(\y)$ be nondecreasing.
Clearly $x_1 = 1$.
We now show that $x_2=1$ by contradiction assuming that $x_2>1$ and considering the following two mutually exclusive possibilities:
\begin{itemize}
    \item[] 
    Case 1: Suppose $x_2 > a$.
    Then, Theorem~\ref{theorem: unique 1} implies $(x_2 - a, x_3 - a) \in \PAINV_2((c,b))$.
    By Theorem~\ref{theorem: invariant pair}, this implies $x_2 = 1 + a$ and either $x_3 = 1 + a$ or $x_3 = 1 + a + b$.
    However, $\x = (1, 1 + a, 1 + a) \notin \PAINV_3(\y)$ since its rearrangement $\x' = (1 + a, 1 + a, 1) \notin \PA_3(\y)$.
    Similarly, $\x = (1, 1 + a, 1 + a + b) \notin \PAINV_3(\y)$  since its rearrangement $\x' = (1 + a, 1 + a + b, 1) \notin \PA_3(\y)$.
    \item[] 
    Case 2: Suppose $x_2 \leq a$.
    If $a = 1$ we have $x_2 = 1$, contradicting the assumption that $x_2 > 1$.
    If $a > 1$, Lemma~\ref{lemma: minentry} implies $\x \notin \PAINV_3(\y)$.
\end{itemize}
Therefore, if $\x \in \PAINV_3(\y)$ is nondecreasing, then $x_1 = x_2 = 1$.
By the proof of Theorem~\ref{theorem: mi triple}, upon substituting $\x = (1, 1, w)$ and $\y = (c, a, b)$ with $a + b > c$, we have that $\x \in \PAINV_3(\y)$ for $w \in \N_{> 1}$ if and only if
\begin{enumerate}[label=(\arabic*)]
    \item $w\leq 1+a+c$ holds, and
    \item at least one of
    \begin{enumerate}[label=(2\alph*)]
        \item $w\leq 1+c$, or
        \item $w=1+b+c$ holds, and
    \end{enumerate}
    \item at least one of
    \begin{enumerate}[label=(3\alph*)]
        \item $w=1+a$, or
        \item $w=1+a+b$ holds.
    \end{enumerate}
\end{enumerate}
(We omit \ref{it: mi triple (3c)} since it requires $a \geq b$, a contradiction.)
To find the valid solutions for $w$, consider the following table, where ``None'' indicates that selecting the conditions marked ``Yes'' leads to no solutions for $w$. 
Namely, in the table, each row corresponds to a combination of ways in which conditions (1), (2), and (3) may hold, and the solutions for $w$ that arise (if any).
Note that empty cells represent that the listed condition is immaterial to that row.
\begin{align*}
    \begin{array}{|c|c|c|c|c|c|}
        \hline
        \text{Cond. (1)} & \text{Cond. (2a)} & \text{Cond. (2b)} & \text{Cond. (3a)} & \text{Cond. (3b)}  & w \\\hline
        \text{Yes} & \text{Yes} & & \text{Yes} & & 1+a \\
        \hline
        \text{Yes} & \text{Yes} & & & \text{Yes} & \text{None} \\
        \hline
        \text{Yes} & & \text{Yes} & \text{Yes} & & \text{None} \\
        \hline
        \text{Yes} & & \text{Yes} & & \text{Yes} & \text{None} \\
        \hline 
    \end{array}
\end{align*}
Therefore, $w = 1 + a$ is the only valid solution satisfying $w > 1$.
This shows that $\PAINVND_3(\y)=\{(1,1,1), (1,1,1+a)\}$.
\end{proof}

\begin{proposition}
\label{prop: y=(c,b,a), a+b<=c}
Let $\y=(c,b,a)\in \N^3$ with $a+b\leq c$. 
Then, $\PAINVND_3(\y)=\{(1,1,1),(1,1,1+a),(1,1,1+b),(1,1,1+a+b)\}$.
\end{proposition}
\begin{proof}
Let $\x = (x_1, x_2, x_3) \in \PAINV_3(\y)$ be nondecreasing.
Clearly $x_1 = 1$.
We now show that $x_2=1$ by contradiction assuming that $x_2>1$ and considering the following two mutually exclusive possibilities:
\begin{itemize}
    \item[] 
    Case 1: Suppose $x_2 > a$.
    Then, Theorem~\ref{theorem: unique 1} implies $(x_2 - a, x_3 - a) \in \PAINV_2((c,b))$.
    By Theorem~\ref{theorem: invariant pair}, this implies $x_2 = 1 + a$ and either $x_3 = 1 + a$ or $x_3 = 1 + a + b$.
    However, $\x = (1, 1 + a, 1 + a) \notin \PAINV_3(\y)$ since its rearrangement $\x' = (1 + a, 1, 1 + a) \notin \PA_3(\y)$.
    Similarly, $\x = (1, 1 + a, 1 + a + b) \notin \PAINV_3(\y)$  since its rearrangement $\x' = (1 + a, 1, 1 + a + b) \notin \PA_3(\y)$.
    \item[] 
    Case 2: Suppose $x_2 \leq a$.
    If $a = 1$ we have $x_2 = 1$, contradicting the assumption that $x_2 > 1$.
    If $a > 1$, Lemma~\ref{lemma: minentry} implies $\x \notin \PAINV_3(\y)$.
\end{itemize}
Therefore, if $\x \in \PAINV_3(\y)$ is nondecreasing, then $x_1 = x_2 = 1$.
By the proof of Theorem~\ref{theorem: mi triple}, upon substituting $\x = (1, 1, w)$ and $\y = (c, b, a)$ with $a + b \leq c$, we have that $\x \in \PAINV_3(\y)$ for $w \in \N_{> 1}$ if and only if
\begin{enumerate}[label=(\arabic*)]
    \item $w\leq 1+b+c$ holds, and
    \item at least one of
    \begin{enumerate}[label=(2\alph*)]
        \item $w\leq 1+c$, or
        \item $w=1+a+c$ holds, and
    \end{enumerate}
    \item at least one of
    \begin{enumerate}[label=(3\alph*)]
        \item $w=1+b$, or
        \item $w=1+a+b$, or
        \item $b\geq a$ and $w=1+a$ holds.
    \end{enumerate}
\end{enumerate}
To find the valid solutions for $w$, consider the following table, where ``None'' indicates that selecting the conditions marked ``Yes'' leads to no solutions for $w$. 
Namely, in the table, each row corresponds to a combination of ways in which conditions (1), (2), and (3) may hold, and the solutions for $w$ that arise (if any).
Note that empty cells represent that the listed condition is immaterial to that row.
\begin{align*}
    \begin{array}{|c|c|c|c|c|c|c|}
        \hline
        \text{Cond. (1)} & \text{Cond. (2a)} & \text{Cond. (2b)} & \text{Cond. (3a)} & \text{Cond. (3b)}&\text{Cond. (3c)}  & w \\\hline
        \text{Yes} & \text{Yes} & & \text{Yes} & & & 1+b \\
        \hline
        \text{Yes} & \text{Yes} & & & \text{Yes} & & 1+a+b \\
        \hline
        \text{Yes} & \text{Yes} & & & & \text{Yes} & 1+a \\
        \hline
        \text{Yes} & & \text{Yes} & \text{Yes} & & & \text{None} \\
        \hline
        \text{Yes} & & \text{Yes} & & \text{Yes} & & \text{None} \\
        \hline
        \text{Yes} & & \text{Yes} & & & \text{Yes} &  \text{None}\\
        \hline 
    \end{array}
\end{align*}
Therefore, $w = 1 + a$, $w = 1 + b$, and $w = 1 + a + b$ are the only valid solutions satisfying $w > 1$.
This shows that $\PAINVND_3(\y)=\{(1,1,1), (1,1,1+a), (1,1,1+b), (1,1,1+a+b)\}$.
\end{proof}

\begin{proposition}
\label{prop: y=(c,b,a), a+b>c}
Let $\y=(c,b,a) \in \N^3$ with $a+b>c$. 
Then, $\PAINVND_3(\y)=\{(1,1,1),(1,1,1+a),(1,1,1+b)\}$.
\end{proposition}
\begin{proof}
Let $\x = (x_1, x_2, x_3) \in \PAINV_3(\y)$ be nondecreasing.
Clearly $x_1 = 1$.
We now show that $x_2=1$ by contradiction assuming that $x_2>1$ and considering the following two mutually exclusive possibilities:
\begin{itemize}
    \item[] 
    Case 1: Suppose $x_2 > a$.
    Then, Theorem~\ref{theorem: unique 1} implies $(x_2 - a, x_3 - a) \in \PAINV_2((c,b))$.
    By Theorem~\ref{theorem: invariant pair}, this implies $x_2 = 1 + a$ and either $x_3 = 1 + a$ or $x_3 = 1 + a + b$.
    However, $\x = (1, 1 + a, 1 + a) \notin \PAINV_3(\y)$ since its rearrangement $\x' = (1 + a, 1, 1 + a) \notin \PA_3(\y)$.
    Similarly, $\x = (1, 1 + a, 1 + a + b) \notin \PAINV_3(\y)$  since its rearrangement $\x' = (1 + a, 1, 1 + a + b) \notin \PA_3(\y)$.
    \item[] 
    Case 2: Suppose $x_2 \leq a$.
    If $a = 1$ we have $x_2 = 1$, contradicting the assumption that $x_2 > 1$.
    If $a > 1$, Lemma~\ref{lemma: minentry} implies $\x \notin \PAINV_3(\y)$.
\end{itemize}
Therefore, if $\x \in \PAINV_3(\y)$ is nondecreasing, then $x_1 = x_2 = 1$.
By the proof of Theorem~\ref{theorem: mi triple}, upon substituting $\x = (1, 1, w)$ and $\y = (c, b, a)$ with $a + b > c$, we have that $\x \in \PAINV_3(\y)$ for $w \in \N_{> 1}$ if and only if
\begin{enumerate}[label=(\arabic*)]
    \item $w\leq 1+b+c$ holds, and
    \item at least one of
    \begin{enumerate}[label=(2\alph*)]
        \item $w\leq 1+c$, or
        \item $w=1+a+c$ holds, and
    \end{enumerate}
    \item at least one of
    \begin{enumerate}[label=(3\alph*)]
        \item $w=1+b$, or
        \item $w=1+a+b$, or
        \item $b\geq a$ and $w=1+a$ holds.
    \end{enumerate}
\end{enumerate}
To find the valid solutions for $w$, consider the following table, where ``None'' indicates that selecting the conditions marked ``Yes'' leads to no solutions for $w$. 
Namely, in the table, each row corresponds to a combination of ways in which conditions (1), (2), and (3) may hold, and the solutions for $w$ that arise (if any).
Note that empty cells represent that the listed condition is immaterial to that row.
\begin{align*}
    \begin{array}{|c|c|c|c|c|c|c|}
        \hline
        \text{Cond. (1)} & \text{Cond. (2a)} & \text{Cond. (2b)} & \text{Cond. (3a)} & \text{Cond. (3b)}&\text{Cond. (3c)}  & w \\\hline
        \text{Yes} & \text{Yes} & & \text{Yes} & & & 1+b \\
        \hline
        \text{Yes} & \text{Yes} & & & \text{Yes} & & \text{None} \\
        \hline
        \text{Yes} & \text{Yes} & & & & \text{Yes} & 1+a \\
        \hline
        \text{Yes} & & \text{Yes} & \text{Yes} & & & \text{None} \\
        \hline
        \text{Yes} & & \text{Yes} & & \text{Yes} & & \text{None} \\
        \hline
        \text{Yes} & & \text{Yes} & & & \text{Yes} &  \text{None}\\
        \hline 
    \end{array}
\end{align*}
Therefore, $w = 1 + a$ and $w = 1 + b$ are the only valid solutions satisfying $w > 1$.
This shows that $\PAINVND_3(\y)=\{(1,1,1), (1,1,1+a), (1,1,1+b)\}$.
\end{proof}

\end{document}